\documentclass[11 pt]{amsart}
\usepackage{amsfonts}
\usepackage{amsmath,amscd}
\usepackage[margin=1in]{geometry}
\usepackage{amssymb}
\usepackage{centernot} 
\usepackage{enumerate} 
\usepackage{parskip}
\usepackage{pb-diagram} 
\usepackage{mathrsfs}
\usepackage[OT2,T1]{fontenc}
\usepackage{seqsplit}
\usepackage{color}
\usepackage{array}
\usepackage{verbatim}
\usepackage{url}

\def\In#1{{\text{\rm I{}$_{#1}$}}}

\def\IIS{\text{\rm II$^*$}}

\DeclareSymbolFont{cyrletters}{OT2}{wncyr}{m}{n}
\DeclareMathSymbol{\Sha}{\mathalpha}{cyrletters}{"58}

\definecolor{refkey}{rgb}{1,1,1}
\definecolor{labelkey}{rgb}{1,1,1}
\definecolor{cite}{rgb}{0.9451,0.2706,0.4941}
\definecolor{ruri}{rgb}{0.0078,0.4022,0.8010}

\usepackage[%
bookmarks=true,bookmarksnumbered=true,%
colorlinks=true,linkcolor=ruri,citecolor=red%
 ]{hyperref}

\def\F{{\rm \mathbb{F}}}
\def\Z{{\rm \mathbb{Z}}}

\def\Q{{\rm \mathbb{Q}}}
\def\QQ{{\rm \mathbb{Q}}}

\def\C{{\rm \mathbb{C}}}

\def\p{{\rm \mathfrak{p}}}

\def\n{{\rm \mathfrak{n}}}
\def\q{{\rm \mathfrak{q}}}
\def\a{{\rm \mathfrak{a}}}

\def\c{{\rm \mathfrak{c}}}

\def\O{{\rm \mathcal{O}}}

\def\Disc{{\rm Disc}}

\def\Gal{{\rm Gal}}

\def\End{{\rm End}}
\def\Sel{{\rm Sel}}

\numberwithin{equation}{section}

\newtheorem{theorem}{Theorem}[section]
\newtheorem{lemma}[theorem]{Lemma}

\newtheorem{remark}[theorem]{Remark}
\newtheorem{definition}[theorem]{Definition}
\newtheorem{example}[theorem]{Example}
\newtheorem{conjecture}[theorem]{Conjecture}
\newtheorem{corollary}[theorem]{Corollary}
\newtheorem{proposition}[theorem]{Proposition}

\usepackage{marginnote,color}
\usepackage[usenames,dvipsnames]{xcolor}

\usepackage{tikz}

\def\shownotes{\def\inline##1##2##3{ \begin{adjustwidth}{3mm}{7mm}\mbox{}\par \noindent
{\color{##1}\hspace{-1.9cm}{\large ##2}\vspace{-\baselineskip}\\##3}
\newline\end{adjustwidth}} \def\inlinewide##1##2##3{ \begin{adjustwidth}{0mm}{0cm}\mbox{}\par \noindent
{\color{##1}\hspace{-1.6cm}{\large ##2}\vspace{-\baselineskip}\\##3}
\newline\end{adjustwidth}}  \def\marg##1##2##3{\marginnote{\color{##1}{\large ##2}\\{\small ##3}}[-.8cm]}}

\shownotes

\begin{document}
\setlength{\parskip}{2pt} 
\setlength{\parindent}{8pt}
\title[Tate-Shafarevich groups]{Elements of given order in Tate-Shafarevich groups of abelian varieties in quadratic twist families}

\author{Manjul Bhargava}
\address{Department of Mathematics, Princeton University, Princeton, NJ 08544}
\email{bhargava@math.princeton.edu}

\author{Zev Klagsbrun}
\address{Center for Communications Research, San Diego, CA 92121}
\email{zdklags@ccrwest.org}

\author{Robert J. Lemke Oliver}
\address{Department of Mathematics, Tufts University, Medford, MA 02155}
\email{robert.lemke{\_{}}oliver@tufts.edu}

\author{Ari Shnidman}
\address{Einstein Institute of Mathematics, Hebrew University, Jerusalem, Israel} 
\email{ariel.shnidman@mail.huji.ac.il}

\maketitle

\begin{abstract}
Let $A$ be an abelian variety over a number field $F$ and let $p$ be a prime.  Cohen-Lenstra-Delaunay-style heuristics predict that the Tate-Shafarevich group $\Sha(A_s)$ should contain an element of order $p$ for a positive proportion of quadratic twists $A_s$ of $A$. We give a general method to prove instances of this conjecture by exploiting independent isogenies of $A$.  For each prime $p$, there is a large class of elliptic curves for which our method shows that a positive proportion of quadratic twists have nontrivial $p$-torsion in their Tate-Shafarevich groups.  In particular, when the modular curve $X_0(3p)$ has infinitely many $F$-rational points the method applies to ``most'' elliptic curves $E$ having a cyclic $3p$-isogeny.  It also applies in certain cases when $X_0(3p)$ has only finitely many points.  For example, we find an elliptic curve over $\mathbb{Q}$ for which a positive proportion of quadratic twists have an element of order $5$ in their Tate-Shafarevich groups.

The method applies to abelian varieties of arbitrary dimension, at least in principle.  As a proof of concept, we give, for each prime $p \equiv 1 \pmod 9$, examples of CM abelian threefolds with a positive proportion of quadratic twists having elements of order $p$ in their Tate-Shafarevich groups.     
\end{abstract}

\section{Introduction}

Let $A$ be an abelian variety over a number field $F$ and let $n$ be a positive integer.  The $n$-Selmer group $\mathrm{Sel}_n(A)$
sits in the short exact sequence
\[
0 \to A(F)/nA(F) \to \mathrm{Sel}_n(A) \to \Sha(A)[n] \to 0
\]
between the (weak) Mordell-Weil group and the $n$-torsion of the Tate-Shafarevich group $\Sha(A)$.  For the definitions of these groups, see \cite[\S X]{AEC}.  Thus, the presence of $n$-torsion in $\Sha(A)$ is the obstruction to computing the rank of the group of rational points $A(F)$ via $n$-descent.  

One is immediately led to ask how often the group $\Sha(A)[n]$ is nontrivial as $A$ varies.  Consider, for example, the quadratic twist  family $A_s$ for $s \in F^*/F^{*2}$, obtained by twisting $A$ by the different quadratic characters of $\Gal(\overline{F}/F)$.  For this family, one conjectures that there is often such an obstruction:

\begin{conjecture}\label{conj:weak-sha}
If $A$ is an abelian variety over a number field $F$ and $n \geq 2$ is a positive integer, then $\Sha(A_s)$ has an element of order $n$ for a positive proportion of $s \in F^*/F^{*2}$, when ordered by height.
\end{conjecture}

A more precise conjecture when $A$ is an elliptic curve has been formulated by Delaunay \cite{Delaunay}, but to date, there is no example of any abelian variety for which even the weaker Conjecture \ref{conj:weak-sha} is known to hold for any $n$ that is not a power of two.  When $n=2$, the only examples of $A$ for which Conjecture \ref{conj:weak-sha} has been established either require $A$ to be an elliptic curve over $\mathbb{Q}$ with full rational two-torsion or for $A$ to admit such a curve as an isogeny factor \cite{FengXiong,Smith,XiongZaharescu},  with Smith's work \cite{Smith} largely confirming the full $2$-power case of Delaunay's conjecture for such curves.
Away from $2$, little is known.  If $n=3,5,$ or $7$ and the elliptic curve $E/\mathbb{Q}$ has a rational $n$-torsion point, then $\Sha(E_s)[n] \neq 0$ for infinitely many twists \cite{BalogOno}, but this result falls far short of obtaining a positive proportion.  It is also known that the $3$- and $5$-parts of $\Sha$ can be arbitrarily large for curves over $\mathbb{Q}$ \cite{Cassels,Fisher}, and that the $p$-part can be arbitrarily large for curves over some number field depending on $p$ \cite{Kloosterman}.

The purpose of the paper at hand is to establish several cases of Conjecture \ref{conj:weak-sha} when $n$ is not a power of two and for abelian varieties other than elliptic curves over $\mathbb{Q}$.  
Our first theorem proves the existence of elliptic curves for which Conjecture \ref{conj:weak-sha} holds with $n=3$.

\begin{theorem}\label{thm:9-isogeny}
Suppose $E \to E'$ is a cyclic $9$-isogeny of elliptic curves over $\Q$.  Then either a positive proportion of the twists $E_s$ have rank $0$ and $|\Sha(E_s)[3]| \geq 9$, or a positive proportion of the twists $E'_s$ have rank $0$ and $|\Sha(E'_s)[3]| \geq 9$.  \end{theorem}

While we enjoy the clean statement of Theorem \ref{thm:9-isogeny}, the method typically applies to both $E$ and $E^\prime$.  This is made clear by the quantitative version, Theorem \ref{lem:9-isog-sha} below. In fact, one consequence of Theorem \ref{lem:9-isog-sha} is that for any fixed number field $F$ and $r \ge 1$, almost all elliptic curves $E$ with a cyclic 9-isogeny defined over $F$ will have a positive proportion of twists $E_s$ with $|\Sha(E_s)[3]| \ge 9^r$.  For some of these curves, this positive proportion may in fact be taken to be a vast majority:

\begin{theorem}\label{thm:rigged-sha}
Let $F$ be a number field and let $r \geq 1$.  For any $\epsilon>0$, there are infinitely many elliptic curves $E/F$, not isomorphic over $\overline{F}$, for which a proportion at least $1-\epsilon$ of twists $E_s/F$ have $|\Sha(E_s)[3]| \geq 9^r$. 
\end{theorem}

The ideas leading to Theorem \ref{thm:9-isogeny} also permit us to find elements of order $6$ in Tate-Shafarevich groups for a positive proportion of twists $E_s$ of $E$, provided that the elliptic curve $E$ has an additional bit of level structure.  For convenience, we state this result only over $\mathbb{Q}$, though a less uniform version should hold over any number field.

\begin{theorem}\label{thm:18-isogeny}
Suppose that $E/\mathbb{Q}$ has a cyclic $18$-isogeny.  Suppose also that $E$ is not a twist of a curve in the isogeny class of the curve $y^2+xy+y=x^3+4x-6$ with Cremona label {\rm 104a1}.  Then for a positive proportion of twists, $\Sha(E_s)$ has an element of order $6$.
\end{theorem}

The isogeny class 14a is the subject of Section \ref{sec:examples}.  While our methods do not show that twists of curves in this isogeny class have elements of order six in their Tate-Shafarevich groups, we are nevertheless able to obtain strong applications regarding their Mordell-Weil ranks.  For example, we prove that at least $25\%$ of their quadratic twists have rank $0$, and at least $41.6\%$ have rank $1$ assuming finiteness of Tate-Shafarevich groups.

For each prime $p \geq 5$, we also provide examples of curves for which Conjecture \ref{conj:weak-sha} holds for $n=p$.  Before stating a general theorem, we recall a bit of notation.  The modular curve $X_0(3p)$ has four cusps that we may label $\mathfrak{c}_1$, $\mathfrak{c}_3$, $\mathfrak{c}_p$, and $\mathfrak{c}_{3p}$ according to their ramification degree over the single cusp of the curve $X(1)$.  Each of these cusps is rational, so if $\mathfrak{p} \nmid 3p$ is a prime ideal in the ring of integers $\mathcal{O}_F$ of some number field $F$, the reduction $\overline{\mathfrak{c}_i} \in X_0(3p)(\mathcal{O}_F/\mathfrak{p})$ is well defined.  With this notation, we have:

\begin{theorem}\label{thm:modular-curve}
Let $F$ be a number field of degree $d$ and suppose that $x \in X_0(3p)(F)$ is a non-cuspidal point.  For $i \in \{1,3,p,3p\}$, let $\omega_i$ denote the number of primes $\mathfrak{p} \nmid 3p$ for which the reduction of $\bar{x} \in X_0(3p)(\mathcal{O}_F/\p)$ is equal to the reduction  $\overline{\mathfrak{c}_i}$ of $\mathfrak{c}_i$.  Then there exists an elliptic curve $E/F$ with $j(E) = j(x)$ such that for a positive proportion of $s$ in the trivial squareclass of $\prod_{v\mid 3p\mathfrak{f}_E\infty} F_v/F_v^{*2}$, we have
\[
|\Sha(E_s/F)[p]| \geq p^{2\min(\omega_1,\omega_3) - 2d}.
\]
\end{theorem}

We present several corollaries to Theorem \ref{thm:modular-curve}, the first of which simply shows that the set of curves for which Theorem \ref{thm:modular-curve} produces a non-trivial result is not empty.

\begin{corollary}\label{cor:p-sha}
Let $p \geq 5$ be a prime and let $E/\mathbb{Q}$ be an elliptic curve.  Let $r \geq 1$ and suppose that $E$ has multiplicative reduction at $\ell \nmid 3p$ for at least $4p+4+r$ primes $\ell$.  Also suppose that $\Gal(\bar \Q/\Q)$ acts transitively on the set of $\F_\ell$-lines in $E[\ell]$, for each $\ell \in \{3,p\}$.  Then there exists a number field $F$ of degree at most $4p+4$ over which $|\Sha(E_s)[p]| \geq p^{2r}$ for a positive proportion of $s \in F^* / F^{* 2}$.  
\end{corollary}

The constants in Theorem \ref{thm:modular-curve} are generally not optimal, and using the ideas behind the proof, we find the first example of an elliptic curve over $\mathbb{Q}$ for which a positive proportion of twists have an element of order $5$ in the associated Tate-Shafarevich group.

\begin{corollary}\label{cor:5-sha}
Let $E\colon y^2 + xy + y = x^3 + x^2 - 13x - 219$ be the elliptic curve with Cremona label $50b3$.  For at least $50\%$ of positive squarefree $s\equiv 1\pmod{8}$ that are coprime to $5$, $E_s(\mathbb{Q})$ has rank $0$ and $|\Sha(E_s)[5]| \geq 25$.  
The same result holds for the elliptic curve with Cremona label $50b4$.
\end{corollary}

Corollary \ref{cor:p-sha} shows that for any $r \geq 1$, there exist some number field over which some elliptic curve $E$ has $|\Sha(E_s)[p]| \geq p^{2r}$ for a positive proportion of its twists.  The final corollary we present shows that, when $p=5$ or $7$, there exist number fields over which for any $r \geq 1$ there are elliptic curves $E$ for which $|\Sha(E_s)[p]| \geq p^{2r}$ for a positive proportion of the quadratic twists of $E$.

\begin{corollary}\label{cor:15-21}
Let $p\in\{5,7\}$ and let $r\geq 1$ be an integer.  Let $F$ be a number field over which the modular curve $X_0(3p)$ has infinitely many points over $F$.  Then there are infinitely many elliptic curves over $F$, not isomorphic over $\bar{F}$, such that $|\Sha(E_s)[p]| \geq p^{2r}$ for a positive proportion of $s \in F^* / F^{*2}$.
\end{corollary}

For example, $X_0(15)$ and $X_0(21)$ both have infinitely many points over the field $\mathbb{Q}(\sqrt{10})$, so each case of Corollary \ref{cor:15-21} applies.  The curves produced by Corollary \ref{cor:15-21} are explicit, in the sense that they can be generated easily in \verb^Magma^.  Moreover, the proof suggests that when the points of $X_0(3p)(F)$ are ordered in a natural way, the conclusion of Corollary \ref{cor:15-21} holds for $100\%$ of curves over $F$ with a degree $3p$ isogeny; see Remark \ref{rem:100}.  

As we explain at the end of this section, the proofs of the above results all make use of primes of multiplicative reduction to force a certain Selmer group to be large.  In particular, this method does not apply to elliptic curves with (potentially) everywhere good reduction.  The next result provides examples of curves with (potentially) everywhere good reduction for which Conjecture~\ref{conj:weak-sha} holds, though as in Corollary \ref{cor:p-sha}, we must base-change to a larger number field to find them.  This second approach, which will also be elaborated on at the end of this section, works for any prime $p\geq 5$:
  
\begin{theorem}\label{large primes}
Let $p \geq 5$ be a prime and $E/\Q$ an elliptic curve with potentially good and ordinary reduction at both $3$ and $p$.  Assume that $\Gal(\bar \Q/\Q)$ acts transitively on the set of $\F_\ell$-lines in $E[\ell]$, for $\ell \in \{3, p\}$, and set $K = \Q(E[3p])$.  Let $E'$ be any elliptic curve over $K$ that is $p$-isogenous to $E$.  Then for a positive proportion of $s \in K^*/K^{*2}$, we have 
\[\dim_{\F_p} \Sha(E'_s)[p] \ge \dfrac d2 \left(1 - \frac{4}{p+1}\right),\] where $[K:\Q] = 2d$.
\end{theorem}

We can sometimes obtain even stronger results in the case of curves with complex multiplication.

\begin{theorem}\label{CMsha}
Let $E$ be an elliptic curve over a number field $F$, and assume that $\End_F(E)$ is the quadratic order of discriminant $Df^2$, with $D$ a fundamental discriminant.  Assume that $f$ is odd, that $3$ is not inert in $K=\mathbb{Q}(\sqrt{D})$, and that all primes dividing $f$ split in $K$.  Then at least half of the twists $E_s$ have both rank $0$ and $|\Sha(E_s)[f]| \geq f^d$, where $[F \colon \Q] = 2d$.  
\end{theorem}

Finally, our methods apply equally well to higher dimensional abelian varieties, though the computations often become more difficult and less explicit.  We provide the following result concerning certain abelian threefolds with CM by the ninth cyclotomic field $K = \Q(\zeta_9)$ as a proof of concept.  

\begin{theorem}\label{picard curve}
Let $J$ be the Jacobian of the Picard curve $y^3=x^4-x$.  Let $p\equiv 1 \pmod{9}$ be a prime, and let $F$ be any number field containing $K(J[p])$.  Then there is an abelian variety $A/F$ isogenous to $J$ such that at least $50\%$ of quadratic twists $A_s$ have rank $0$ and satisfy $|\Sha(A_s)[p]|\geq p^{3d}$, where $[F:\mathbb{Q}]=2d$.
\end{theorem} 

Theorem \ref{picard curve} has already inspired other results for high-dimensional $A$, and even over $\Q$.  For example the fourth author proved the $n = 3$ case of Conjecture \ref{conj:weak-sha} for certain quotients $A$ of prime level modular Jacobian $J_0(p)$.  The explicit proportion of twists with $\Sha(A_s)[3] \neq 0$ is shown to be at least $1/8$ for these $A$; see \cite[Thm.\ 1.5d]{ari}.  In forthcoming work, Bruin, Flynn, and Shnidman give an explicit  three-parameter family of abelian surfaces over $\Q$ for which one can prove Conjecture \ref{conj:weak-sha} as well.

We now sketch an outline of the method used to prove the above theorems.  The proofs all follow the same general strategy, namely, to exploit abelian varieties that have two independent isogenies.  To any isogeny $\phi\colon A \to A^\prime$ of abelian varieties, we attach a Selmer group $\mathrm{Sel}_{\phi}(A)$.  This Selmer group parameterizes $\phi$-coverings of $A$, up to isomorphism, and sits in the exact sequence
\begin{equation}\label{selmersequence}
0 \to A^\prime(F)/\phi(A(F)) \to \mathrm{Sel}_\phi(A) \to \Sha(A)[\phi] \to 0.
\end{equation}

In favorable circumstances, the $\phi$-Selmer group provides some measure of control over the rank of $A(F)$.  However, when $A$ has two independent isogenies $\phi_1$ and $\phi_2$, the control provided by the two associated Selmer groups need not be the same.  We exploit this imbalance to prove our theorems.

More specifically, in the next section, we define the \emph{global Selmer ratio} $c(\phi_s)$ attached to the quadratic twists $\phi_s$ of an isogeny $\phi$, which has the property that
\[
|\mathrm{Sel}_{\phi_s}(A_s)| \geq c(\phi_s)
\]
for all but finitely many twists $s$.  When $\phi$ has odd degree, $c(\phi_s)$ is determined by finitely many local conditions.   If moreover $\phi$ is not self-dual, then it is relatively easy to construct quadratic twists for which $\mathrm{Sel}_{\phi_s}(A_s)$ is large by making $c(\phi_s)$ large.\footnote{If $\phi$ is self-dual, e.g.\ multiplication by $n$ on an elliptic curve, then $c(\phi) = 1$.  This explains why we need level structure to make this strategy work.}  On the other hand, when $\phi$ has degree $3$ and $A$ has dimension one, recent work of the authors \cite{BKLOS} shows that
\[
\mathrm{Avg}_{s\in \Sigma} \, |\mathrm{Sel}_{\phi_s}(A_s)| = 1 + \mathrm{Avg}_{s\in \Sigma}\, c(\phi_s)
\]
for any $\Sigma \subseteq F^*/F^{*2}$ defined by finitely many local conditions.  In particular, if $c(\phi_s)$ is small for $s \in \Sigma$, then the Selmer group, and hence the rank, is small on average.  Our approach is to use the Selmer groups attached to such isogenies to control the rank, while simultaneously finding an independent isogeny $\psi$ whose associated Selmer group is large.  The sequence \eqref{selmersequence} will then imply that $\Sha(A_s)$ is often non-trivial.   

As alluded to earlier, we have two different ways of constructing twists with large global Selmer ratios $c(\phi_s)$, and hence large Selmer groups.  
The first is to consider twists with many primes of split multiplicative reduction; see Sections \ref{9-isogeny}--\ref{sec:modular-curves}.  The other systematic way we have of making $c(\phi_s)$ large is by exploiting the Galois action on the canonical subgroups of $A[p]$.  This is our approach in Sections \ref{arb p}--\ref{ab vars}.  The analysis is easier when $A$ has ordinary reduction at $p$ and $3$, as in Theorem \ref{large primes}.  In principle, this approach should work in some cases of supersingular reduction as well.        



\section{Global Selmer ratios and Selmer groups}
\label{sec:global_selmer_ratios}
Given an isogeny $\phi\colon A \to A^\prime$ over a number field $F$, the \emph{local Selmer ratio} $c_\p(\phi)$ at a (possibly infinite) place $\p$ is defined to be
\begin{equation}
\label{eq:local_ratio_def}
c_\p(\phi) := \frac{|A^\prime(F_\p)/\phi(A(F_\p))|}{|A(F_\p)[\phi]|}.
\end{equation}
We have $c_\p(\phi)=1$ for all but finitely many primes $\p$, so we may define the \emph{global Selmer ratio} $c(\phi)$ to be the product of the local Selmer ratios, 
\begin{equation}
\label{eq:global_ratio_def}
c(\phi) := \prod_{\p \leq \infty} c_\p(\phi).
\end{equation}
The following proposition records the connection between the Selmer ratio $c(\phi)$ and the two Selmer groups $\mathrm{Sel}_\phi(A)$ and $\mathrm{Sel}_{\hat\phi}(\hat A')$ coming from $\phi$ and the dual isogeny $\hat \phi \colon \hat A' \to \hat A$.

\begin{proposition}\label{prop:selmer-ratio}
Let $\phi\colon A \to A^\prime$ be an isogeny of abelian varieties over a number field $F$.  Then $$\displaystyle c(\phi) = \frac{|\mathrm{Sel}_\phi(A)|}{|\mathrm{Sel}_{\hat\phi}(\hat A^\prime)|}\frac{|A(F)[\phi]|}{|\hat A^\prime(F)[\hat\phi]|}.$$ 
\end{proposition}
\begin{proof}
See Theorem VIII.7.9 in \cite{neukirch}, for example.
\end{proof}

If $\phi$ has prime degree $\ell$, then $c(\phi) =\ell^m$ for some $m \in \mathbb{Z}$. Moreover, if $A$ is an elliptic curve the parity of $m$ encodes information about the rank of the $\ell$-Selmer group:

\begin{proposition}
\label{prop:tamagawaparity}
If $\phi:E\rightarrow E^\prime$ is an isogeny of elliptic curves of prime degree $\ell$ and $c(\phi)=\ell^m$, then $$\dim_{\F_\ell} \mathrm{Sel}_\ell(A) \equiv m + \dim_{\F_\ell} E(F)[\ell] \pmod{2}.$$
\end{proposition}

\begin{proof}
This can be deduced from results of Cassels \cite{Cassels8}; for a proof see  \cite[Prop.\ 42]{j=0}. 
\end{proof}

For each $s \in F^\times/F^{\times2}$, the twist of $\phi$ is an isogeny $\phi_s \colon A_s \to A^\prime_s$ between the quadratic twists. There are associated local Selmer ratios $c_\p(\phi_s)$ for each place $\p$ of $F$ and a global Selmer ratio $c(\phi_s)$.

\begin{corollary}\label{cor:selmer-size}
If $\phi$ has odd degree,  then $|\mathrm{Sel}_{\phi_s}(A_s)| \geq c(\phi_s)$ for all but finitely many $s$.
\end{corollary}
\begin{proof}
If $\phi$ has odd degree, then there at most a single class $s \in F^\times/F^{\times2}$ with $E_s(F)[\phi_s] \ne 0$ and at most a single class $s^\prime \in F^\times/F^{\times2}$ with $E^\prime_{s^\prime}(F)[\hat\phi_{s^\prime}] \ne 0$. By Proposition \ref{prop:selmer-ratio}, we therefore have $c(\phi_s) = \frac{|\mathrm{Sel}_{\phi_s}(A_s)|}{|\mathrm{Sel}_{\hat\phi_s}(A^\prime_s)|}$ for all but finitely many $s$.
\end{proof}

If $\phi$ decomposes as the composition of isogenies $\phi_2 \circ \phi_1$, then Lemma 7.2(b) in \cite{milne} shows that $c_p(\phi) = c_p(\phi_1) c_p(\phi_2)$ for every prime $p$. It therefore follows from \eqref{eq:global_ratio_def} that $c(\phi) = c(\phi_1)c(\phi_2)$. As a result, we may always reduce the computation of Selmer ratios to those of isogenies of prime degree,
in which case, the following result gives a way to compute the local Selmer ratio.

\begin{proposition}\label{prop:local-selmer-comp}
Let $\phi\colon A \to A^\prime$ be an isogeny of prime degree $\ell$.  If $\p$ is a finite prime, then
\[
c_\p(\phi) = \frac{c_\p(A^\prime)}{c_\p(A)} \alpha_{\phi,\p}
\]
where $c_\p(A)$ and $c_\p(A^\prime)$ are the local Tamagawa numbers of $A$ and $A^\prime$ at $\p$ and $\alpha_{\phi, \p}^{-1}$ is the normalized $\p$-adic absolute value of the determinant of the Jacobian matrix of $\phi$ evaluated at the origin; in particular, $\alpha_{\phi, \p}=1$ if $\p \nmid \ell$.  If $\p$ is an infinite place and $\ell$ is odd, then
\[
c_\p(\phi) = \left\{\begin{array}{ll} 1/\ell, & A[\phi] \subseteq A(F_\p) \\ 1, & A[\phi] \not\subseteq A(F_\p). \end{array}\right.
\]
\end{proposition}
\begin{proof}
These statements can all be found in \cite[\S 3]{Schaefer}.
\end{proof}

\begin{remark}{\em
If $\p \nmid \ell$ is a prime, then $\alpha_{\phi, \p}=1$. As a result, we have $c_\p(\phi) = \frac{c_\p(A^\prime)}{c_\p(A)}$. If $A$ has good reduction, we then have $c_\p(\phi) =1$ and the fact that $c_\p(\phi) = 1$ for all but finitely many primes $\p$ (noted at the beginning of this section) follows immediately.} 
\end{remark}

\begin{remark}{\em
When $A$ is an elliptic curve, the factor $\alpha_{\phi, \p}$ is simply the inverse absolute value of the linear term in the power series giving the induced isogeny on formal groups.} 
\end{remark} 

In the special case that $A$ is an elliptic curve, we are able to explicitly compute $c_\p(\phi)$ using the $j$-invariants of $E$ and $E^\prime$ at places of (potential) multiplicative reduction.

\begin{lemma}\label{lem:split-mult}
Let $\phi\colon E \to E^\prime$ be an isogeny of elliptic curves with odd prime degree $\ell$ and suppose that $E$ has $($potentially$)$ multiplicative reduction at a prime $\p\nmid \ell$.
\begin{enumerate}[(i)]
\item If $E$ has split multiplicative reduction at $\p$, then $c_\p(\phi) = v_\p (j(E')) / v_\p( j(E))$.
\item If $E$ does not have split multiplicative reduction at $\p$, then $c_\p(\phi)=1$.
\end{enumerate}
\end{lemma}
\begin{proof}
This follows from Table 1 in \cite{DD}.
\end{proof}

The formulas are even simpler in the case where $A$ has a quadratic twist of good reduction.

\begin{lemma}
\label{lem:twistofgoodred}
Let $\phi\colon A \to A^\prime$ be an isogeny of odd degree $d$.  If $\p \nmid d$ is a prime such $A_s$ has good reduction at $\p$ for some $s \in F^*/F^{*2}$, then $c_\p(\phi) = 1$.
\end{lemma} 
\begin{proof}
This follows from Lemma 4.6 in \cite{DD} and the fact that $d$ is odd. While the result in \cite{DD} is only stated for elliptic curves, its proof nonetheless holds verbatim for abelian varieties of arbitrary dimension.
\end{proof}

As a consequence of Lemma \ref{lem:twistofgoodred}, we deduce:

\begin{corollary}\label{cor:loc-const}
Let $\phi\colon A \to A^\prime$ be an isogeny of odd degree $d$.  For any $s\in F^*/F^{*2}$, the value of $c(\phi_s)$ depends only on the class of $s \in \prod_{p\mid d\Delta_E \infty} F_\p^*/F_\p^{*2}$ and in particular is given by
\[
c(\phi_s) = \prod_{\p\mid d\Delta_A \infty} c_\p(\phi_s).
\]
\end{corollary}
\begin{proof}
If $\p\nmid d\Delta_A \infty$, then $A_s$ and $A'_s$ have quadratic twists of good reduction at $\p$.   By Lemma \ref{lem:twistofgoodred}, we therefore have $c_\p(\phi_s) = 1$ for all such primes. 
\end{proof}

When $\phi$ is of odd prime degree $\ell$, Corollary \ref{cor:loc-const} implies that the sets
\[
T_m(\phi) := \{ s \in F^*/F^{*2} : c(\phi_s) = \ell^m\}
\]
for $m\in\mathbb{Z}$ are defined by finitely many local conditions in the sense of \cite{BKLOS}.  In particular, they have positive density within $F^*/F^{*2}$ when they are non-empty.  Moreover, Corollary \ref{cor:selmer-size} shows that for all but finitely many $s \in T_m(\phi)$, we have the following lower bound on the size of the $\phi$-Selmer group: $|\mathrm{Sel}_\phi(A_s)| \geq \ell^m$.  

On the other hand, when $\ell=3$, the main results of \cite{BKLOS} allow us to control the average size of $\mathrm{Sel}_\phi(A_s)$ for $s\in T_m(\phi)$ and give an {\it upper bound} on its average rank:

\begin{theorem}\label{thm:selmer-average}
Let $\phi\colon A \to A^\prime$ be an isogeny of degree $3$ and for $m\in \mathbb{Z}$, let $T_m(\phi)$ be defined as above.  For any non-empty subset $\Sigma \subseteq T_m(\phi)$ defined by finitely many local conditions, the average size of $\mathrm{Sel}_\phi(A_s)$ for $s\in \Sigma$ is exactly $1 + 3^m$.  Moreover, if $A$ is an elliptic curve, the average rank of $A_s(F)$ for $s \in \Sigma$ is at most $|m| + 3^{-|m|}$.
\end{theorem}
\begin{proof}
This is a combination of Theorems 2.1 and 2.4 in \cite{BKLOS}.
\end{proof}

In particular, when $A=E$ is an elliptic curve, the upper bound on the average rank of the twists $E_s(F)$ for $s\in T_m(\phi)$ implies that a positive proportion of $s \in T_m(\phi)$ have rank at most $|m|$.  Pulling together all of the results of this section, we obtain the following key proposition.

\begin{proposition}\label{prop:sha-elements}
Let $E$ be an elliptic curve over a number field $F$.  Suppose that $\phi\colon E \to E^\prime$ and $\psi\colon E \to E^{\prime\prime}$ are isogenies over $F$ of degree $3$ and degree $\ell$ prime respectively and further assume that $E[\phi] \cap E[\psi] = 0$ in the special case that $\ell =3$.  Let $m$ and $n$ be integers such that $m > |n| \geq 0$ and for which $T_{-m}(\psi) \cap T_n(\phi) \neq \emptyset$.  Then a positive proportion of $s \in T_{-m}(\psi) \cap T_n(\phi)$ have $|\Sha(E'_s)[\ell]| \geq \ell^{m-n}$.
\end{proposition}

In the special case that $\ell = 3$, the discrepancy between the two Selmer groups allows us to strengthen the rank bounds provided by Theorem \ref{thm:selmer-average}.

\begin{corollary}\label{cor:ranks}
Let $\phi\colon E \to E^\prime$ and $\psi\colon E \to E^{\prime\prime}$ be independent $3$-isogenies.  If $T_m(\phi) \cap T_n(\psi)$ is non-empty, then the average rank of $E_s(F)$ for $s \in T_m(\phi) \cap T_n(\psi)$ is at most 
\[\min(|m|,|n|)+3^{-\min(|m|,|n|)}.\]
\end{corollary}
\begin{proof}
We apply Theorem \ref{thm:selmer-average} with whichever choice of $\phi$ and $\psi$ yields a strong bound.
\end{proof}

\begin{remark} \rm
We give examples of elliptic curves for which Corollary \ref{cor:ranks} implies an improved rank bound in Section \ref{sec:examples}.
\end{remark}



\section{Local Selmer ratios for curves with two independent $3$-isogenies}\label{sec:two-isogenies}

Let $F$ be a number field and $E$ an elliptic curve with a pair of independent 3-isogenies $\phi_1:E \rightarrow E^\prime$ and $\phi_2:E \rightarrow E^{\prime\prime}$ over $F$. 
By the non-degeneracy of the Weil pairing, $E$ is a quadratic twist of a curve $E_s$ with $E_s[3] \simeq \Z/3\Z \times \mu_3$.  This immediately implies:
\begin{lemma}
\label{lem:duality}
Suppose $E$ has two independent $3$-isogenies $\phi_1:E \rightarrow E^\prime$ and $\phi_2:E \rightarrow E^{\prime\prime}$.  Then
\begin{enumerate}[$(i)$]
\item Some twist of $E$ obtains full rational $3$-torsion over $F(\zeta_3)$.
\item The groups $E[\phi_1]$ and $E[\phi_2]$ are Cartier dual.
\end{enumerate}
\end{lemma}

Part (i) of Lemma \ref{lem:duality} has the following important corollary.

\begin{lemma}
\label{lem:9isog good twist}
If $E$ has additive, potentially good reduction at a finite prime $\p \nmid 3$ of $F$, then $E$ has a twist $E_s$ with good reduction at $\p$.  In particular, $c_\p(\phi_i) = 1$ for $i = 1,2$. 
\end{lemma}
\begin{proof}
By Lemma \ref{lem:duality}(i), $E$ has some twist $E_s$ with full rational 3-torsion over the extension $F_\p(\zeta_3)$.  It therefore has good reduction over $F_\p(\zeta_3)$ \cite[\S2]{serre-tate}. Since $F_\p(\zeta_3)$ is unramified over $F_\p$, $E_s$ must have good reduction over $F_\p$.   It then follows from Lemma \ref{lem:twistofgoodred} that $c_\p(\phi_i) = 1$. 
\end{proof}

The story when $E$ has multiplicative or additive potential multiplicative reduction at $\p$ is relatively straightforward as well.  For this, we use the {\it Hesse model}.

\begin{lemma}\label{lem:9-isogeny-param}
If $E$ is an elliptic curve over $F$ with $E[3] \simeq \Z/3\Z \times \mu_3$, then there are $u$ and $v$ in $K$, with $v/u \notin \{0,1, \zeta_3, \zeta_3^2\}$ such that $E$ is isomorphic to 
\[
E_{u,v}\colon v(x^3 + y^3 + z^3) = 3uxyz.
\]
In this model, $E_{u,v}(\overline{F})[3]$ is generated by $(1:-1:0)$ and $(\zeta_3:-\zeta_3^2:0)$. Moreover, if $E'$ and $E''$ are the quotients of $E$ by the corresponding subgroups of order $3$, we have
\begin{align*}
j(E') &= \frac{27u^3(9u^3 - 8v^3)^3}{v^9(u-v)(u^2+uv+v^2)}, \\ \smallskip
j(E) &= \frac{27u^3(u+2v)^3(u^2-2uv + 4v^2)^3}{v^3(u-v)^3(u^2+uv+v^2)^3}, \quad \text{ and} \\ \smallskip
j(E^{\prime\prime}) &= \frac{3(u+2v)^3(u^3 + 78u^2v + 84uv^2 + 80v^3)^3}{v(u-v)^9(u^2+uv+v^2)}.
\end{align*}
\end{lemma}

\begin{proof}
The Hesse model is classical (see \cite{RubinSilverberg}, e.g.), and the rest follows from a symbolic computation in \verb^Magma^. 
\end{proof}

\begin{lemma}
\label{lem:9isog mult vals}
Suppose that $E$ has $($potential$)$ multiplicative reduction at $\p$, and write $v(\cdot)$ for the $\p$-valuation.  Let $j$, $j^\prime$, and $j^{\prime \prime}$ denote the $j$-invariants of $E$, $E^\prime$, and $E^{\prime\prime}$, respectively.  Then one of the following holds:
\begin{enumerate}[$(i)$]
\item $v(j) = 3v(j^\prime) $ and $3v(j) = v(j^{\prime\prime})$, 
\item $v(j) = 3v(j^\prime) $ and $v(j) = 3v(j^{\prime\prime})$, or 
\item $3v(j) = v(j^\prime) $ and $v(j) = 3v(j^{\prime\prime})$. 
\end{enumerate}
Moreover, $(ii)$ can only occur if $\mu_3 \subset F_\p$. 
\end{lemma}
\begin{proof}
That only the listed possibilities occur follows from examining the denominators of the $j$-invariants of $E$, $E^\prime$, and $E^{\prime\prime}$ in the Hesse model. This examination also shows that the case $v(j) = 3v(j^\prime) $ and $v(j) = 3v(j^{\prime\prime})$ occurs if and only if $v(t^2 + t + 1) \ne 0$, which implies $\mu_3 \subset F_\p$.
\end{proof}

For any 3-isogeny $\phi$, define the global log-Selmer ratio $t(\phi)$ and the local log-Selmer ratio $t_p(\phi)$, by $c(\phi) = 3^{t(\phi)}$ and $c_p(\phi) = 3^{t_p(\phi)}$.  Thus, $t(\phi) = \sum_{p \leq \infty} t_p(\phi)$.  
 
\begin{lemma}\label{ratios mod 2}
If $F= \Q$, then
\begin{enumerate}[$(i)$]
\item \label{cong:1}$t(\phi_1) \equiv t(\phi_2) \pmod{2}$ and 
\item \label{cong:2}$t_p(\phi_1) \not\equiv t_p(\phi_2) \pmod{2}$ for $p \in \{3, \infty\}$.  
\end{enumerate}
\end{lemma}

\begin{proof}
By Proposition \ref{prop:tamagawaparity}, both sides of Congruence \eqref{cong:1} are equal to the parity of $\dim_{\F_3} \Sel_3(E) -\dim_{\F_3}E(\QQ)[3]$ and hence must be the same.

Lemma \ref{lem:twistofgoodred} and Proposition \ref{prop:local-selmer-comp} show that $t_p(\phi_1) \equiv t_p(\phi_2) \pmod 2$ for all primes $p \nmid 3$ where $E$ has additive, potentially good reduction and Lemmas \ref{lem:split-mult} and  \ref{lem:9isog mult vals} show that $t_p(\phi_1) \equiv t_p(\phi_2) \pmod 2$ for all prime $p\nmid 3$ where $E$ has (potential) multiplicative reduction. Congruence \eqref{cong:1} then gives \begin{equation}\label{eq:t3plustinfty}t_3(\phi_1) + t_\infty(\phi_1) \equiv t_3 + t_\infty(\phi_2) \pmod 2.\end{equation} 

However, by Proposition \ref{prop:local-selmer-comp} combined with Lemma \ref{lem:duality}, we have $t_\infty(\phi_1) \not\equiv t_\infty(\phi_2) \pmod 2$.  Combining this with \eqref{eq:t3plustinfty}, 
we then get that $t_3(\phi_1) \not\equiv t_3(\phi_2) \pmod 2$. 
\end{proof}

Let $\phi_{1,s} \colon E_s \to E'_s$ and $\phi_{2,s} \colon E_s \to E''_s$ denote the isogenies on the twists induced by $\phi_1$ and $\phi_2$.

\begin{lemma}\label{lem:two-isogenies-trivial}
Suppose $F = \Q$.  Then for each prime $p\neq 3$, there exists $s \in \mathbb{Q}_p^*/\mathbb{Q}_p^{*2}$ such that $c_p(\phi_{1,s}) = c_p(\phi_{2,s}) = 1$.  There also exists some $s \in \Q_3^\times/\Q_3^{\times2}$ for which there is an equality of sets
\[\{c_3(\phi_{1,s}), c_3(\phi_{2,s})\} = \{1,3\}.\] 
\end{lemma}

\begin{proof}
First assume $p \neq 3$.  There is a unique $s \in \Q_p^\times/\Q_p^{\times2}$ such that $E[\phi_{1,s}]$ has a $\Q_p$-point, and similarly for $\phi_2$.  Thus, we can choose $s$ so that both $E[\phi_{1,s}]$ and $E[\phi_{2,s}]$ have no $\Q_p$-points.  Then by part (ii) of Lemma \ref{lem:duality}, the groups $E''[\hat\phi_{2,s}]$ and $E'[\hat\phi_{1,s}]$ also have no $\Q_p$-points.  By \eqref{eq:local_ratio_def}, all four ratios $c_p(\hat\phi_{i,s})$ and $c_p(\hat\phi_{i,s})$ are therefore positive integers.  Since  
\[c_p(\phi_{i,s})c_p(\hat\phi_{i,s}) = \dfrac{c_p(E_i)}{c_p(E)}\dfrac{c_p(E)}{c_p(E_i)} = 1,\]
for $i = 1,2$, we conclude that all four local Selmer ratios are equal to 1, as desired.  

Since $\alpha_{\phi_{i,s},\p} \alpha_{\hat \phi_{i,s},\p} = \alpha_{[3],\p} = 3$, we have $$c_3(\phi_{1,s}) c_3(\hat \phi_{1,s}) = 3 = c_3(\phi_{1,s}) c_3(\hat \phi_{1,s}).$$ By choosing $s \in \Q_3^\times/(\Q_3^\times)^2$ such that all four local Selmer ratios are positive integers,  we therefore get that one of $c_3(\phi_{1,s})$ and $c_3(\hat \phi_{1,s})$ is $1$ and the other is $3$. The same holds for  $c_3(\phi_{2,s})$ and $c_3(\hat \phi_{2,s})$. Since $c_3(\phi_{1,s}) \neq c_3(\phi_{2,s})$ by part \eqref{cong:2} of Lemma \ref{ratios mod 2}, one of $c_3(\phi_{1,s})$ and $c_3(\phi_{2,s})$ is 1 and the other is 3.  
\end{proof}

Over general number fields $F$, similar arguments give the following weak version of Lemma \ref{lem:two-isogenies-trivial}:

\begin{lemma}\label{lem:optimal-twist-nf}
Let $F$ be a number field.  Then for each $\p \nmid 3$, there exists $s \in F^*/F^{*2}$ such that $c_\p(\phi_{1,s})=c_\p(\phi_{2,s})=1$.  There also exists $s \in F^*/F^{*2}$ such that
\[
\prod_{\p\mid 3} c_\p(\phi_{1,s}) \in \{1,3,\dots,3^d\} \quad \text{\rm and} \quad \prod_{\p\mid 3} c_\p(\phi_{2,s}) \in \{1,3,\dots,3^d\}.
\]
\end{lemma}



\section{Curves with a $9$-isogeny}\label{9-isogeny}

In this section, we prove Theorems \ref{thm:9-isogeny} and \ref{thm:rigged-sha}.  Let $\phi\colon E' \to E^{\prime\prime}$ be a cyclic $9$-isogeny over $F$.  Then $\phi$ factors as a composition of two $3$-isogenies
 \[E' \stackrel{\hat\phi_1}{\longrightarrow} E \stackrel{\phi_2}{\longrightarrow} E''\]
 over $F$.
The intermediate elliptic curve $E$ has two independent 3-isogenies, $\phi_1$ and $\phi_2$, so it is subject to Proposition \ref{prop:sha-elements} and the results of Section \ref{sec:two-isogenies}.

\subsection{Proof of Theorem \ref{thm:9-isogeny}}
Assume $F=\Q$.  We may replace $E', E,$ and $E''$ with their quadratic twists, so we may assume that $E[3] \simeq \Z/3\Z \times \mu$.  By Lemma \ref{lem:9-isogeny-param}, $E$ is isomorphic to 
\[E_{u,v}\colon v(x^3 + y^3 + z^3) = 3uxyz,\]
for some coprime integers $u$ and $v$, with $v/u \notin \{0,1\}$. 

We first consider the case where there is some prime $\ell \geq  5$ dividing $u^2 + uv  + v^2$.  We claim that there exists $s \in \Q^\times/\Q^{\times2}$ such that

\begin{enumerate}[(i)]
\item \label{eq:loc_ell} $c_\ell(\phi_{1,s}) = \frac{1}{3} = c_\ell(\phi_{2,s})$,
\item \label{eq:loc_p}$c_p(\phi_{1,s}) = 1 = c_p(\phi_{2,s})$ for all $p \nmid 3\ell\infty$, and
\item \label{eq:loc_3infty} One of $c_3(\phi_{i,s})c_\infty(\phi_{i,s})$, for $i \in \{1,2\}$, equals 3 and the other equals $\frac13$.
\end{enumerate}

By weak approximation, it suffices to shows that each of these can be individually satisfied, and  indeed, \eqref{eq:loc_ell} follows from Lemmas \ref{lem:9-isogeny-param}, \ref{lem:9isog mult vals}, and Lemma \ref{lem:split-mult}, \eqref{eq:loc_p} follows from Lemma \ref{lem:two-isogenies-trivial}, and \ref{eq:loc_3infty} follows from combining Proposition \ref{prop:local-selmer-comp} with Lemma \ref{lem:two-isogenies-trivial}.

It follows that at least one of the sets $T_{-2}(\phi_1) \cap T_0(\phi_2)$ and $T_{0}(\phi_1) \cap T_0(\phi_{-2})$ is non-empty.  Theorem \ref{thm:9-isogeny} now follows from Proposition \ref{prop:sha-elements}.  

We next consider the case where $q := u^2 + uv + v^2$ is not divisible by any prime $\ell > 3$.  Since $q$ is positive definite as a quadratic form in $(u,v)$, we have $q > 0$ and since $u$ and $v$ are coprime, $q$ can't be divisible by either $2$ or $9$, 
so it follows that $q \in \{1,3\}$.  Up to quadratic twist, these cases correspond to $E$ with Cremona labels $27a1$ and $54a1$. A computation in \verb^Magma^ shows that $T_{-2}(\phi_1) \cap T_0(\phi_2)$ is non-empty and Theorem \ref{thm:9-isogeny} therefore follows from Proposition \ref{prop:sha-elements} as above.

\subsection{Quantitative version of Theorem \ref{thm:9-isogeny} over number fields}

Assume now that $F$ has degree $d$ over $\Q$, with $r_1$ real places and $r_2$ pairs of complex places.  Define sets $\mathcal{P}_1$, $\mathcal{P}_2$, and $\mathcal{P}_3$ by
\begin{align*}
\mathcal{P}_1 &= \{\p : E \text{ has (pot.) mult. red. at $\p$ with } v(j) = 3v(j^\prime) \text{ and } 3v(j) = v(j^{\prime\prime}) \\ 
\mathcal{P}_2 &=  \{\p : E \text{ has (pot.) mult. red. at $\p$ with } 
v(j) = 3v(j^\prime)\text{ and } v(j) = 3v(j^{\prime\prime}). 
\end{align*}
and
\[
\mathcal{P}_3 =\{\p : E \text{ has (pot.) mult. red. at $\p$ with } 3v(j) = v(j^\prime) \text{ and } v(j) = 3v(j^{\prime\prime})  
\]

Also set $\mathcal{P}_\mathrm{mult}=\mathcal{P}_1\cup\mathcal{P}_2\cup\mathcal{P}_3$.  The following theorem gives general results about $3$-torsion elements in Tate-Shafarevich groups of quadratic twists for elliptic curves with a 9-isogeny over $F$.

\begin{theorem}\label{lem:9-isog-sha}
Let $s_0 \in F^*/F^{*2}$ be any square class such that $E_{s_0}$ does not have split multiplicative reduction at any $\p \in \mathcal{P}_3$.  Let $\omega_1^\mathrm{sp}$ denote the number of $\p \in \mathcal{P}_1$ at which $E_{s_0}$ has split multiplicative reduction, and define $\omega_2^\mathrm{sp}$ analogously.  Then for a positive proportion of $s\in s_0 F^{*2}$
\[
\mathrm{dim}_{\mathbb{F}_3} \Sha(E_s)[3] \geq \min(2\omega_1^\mathrm{sp}-2r_2,2\omega_2^\mathrm{sp}-d).
\]
\end{theorem}
\begin{proof}
By Lemma \ref{lem:9isog good twist}, the only places that contribute to either of the Selmer ratios $c(\phi_1)$ and $c(\phi_2)$ are infinite primes, places of (potentially) multiplicative reduction, and primes $\p \mid 3$.  For any $s\in F^*/F^{*2}$, define $t_1 = \mathrm{ord}_3(c(\phi_{1,s}))$ and $t_2 = \mathrm{ord}_3(c(\phi_{2,s}))$, along with 
\[
t_{1,\mathrm{mult}} = \mathrm{ord}_3\prod_{\p \in \mathcal{P}_\mathrm{mult}} c_\p(\phi_{1,s}) \quad\text{ and }\quad t_{2,\mathrm{mult}} = \mathrm{ord}_3 \prod_{\p \in \mathcal{P}_\mathrm{mult}} c_\p(\phi_{2,s}).
\]
From Lemma \ref{lem:split-mult} and Lemma \ref{lem:9isog mult vals}, it follows that for any $s\in s_0F^{*2}$ 
\[
t_{1,\mathrm{mult}} = {-\omega_1^\mathrm{sp}-\omega_2^\mathrm{sp}} \quad \text{ and }\quad  t_{2,\mathrm{mult}} = {\omega_1^\mathrm{sp}-\omega_2^\mathrm{sp}}.
\]
It remains to consider the contribution of places $\p\mid 3\infty$.  By Proposition \ref{prop:local-selmer-comp}, we have
\[
\prod_{\p \mid \infty} c_\p(\phi_{1,s}) = 3^{-r_1-r_2} \quad \text{and} \quad \prod_{\p \mid \infty} c_\p(\phi_{2,s}) = 3^{-r_2}.
\] 
for the positive proportion of $s \in s_0 F^{*2}$ such that $E[\phi_{i,s}](K_v) = \Z/3\Z$ for all complex places $v$ and $E[\phi_{i,1}](K_v) = \Z/3\Z$ and $E[\phi_{i,2}](K_v) = 0$ for all real places.

Appealing to Lemma \ref{lem:optimal-twist-nf}, we find for all such $s$, we have
\[
t_1 \in [-\omega_1^\mathrm{sp}-\omega_2^\mathrm{sp}-d+r_2,-\omega_1^\mathrm{sp}-\omega_2^\mathrm{sp}+r_2]
\]
and
\[
t_2 \in [\omega_1^\mathrm{sp}-\omega_2^\mathrm{sp}-r_2,\omega_1^\mathrm{sp}-\omega_2^\mathrm{sp}-r_2+d].
\]
It then follows that
\[
|t_1|-|t_2| \geq \min(2\omega_1^\mathrm{sp}-2r_2, 2\omega_2^\mathrm{sp} - d),
\]
so that the claim now follows from Proposition \ref{prop:sha-elements}.
\end{proof}

\subsection{Proof of Theorem \ref{thm:rigged-sha}}

For this theorem, we wish to construct elliptic curves $E/F$ for which a vast majority of quadratic twists the group $\Sha(E_s)[3]$ is large.  We first loosely discuss the idea, continuing with the notation of the proof of Lemma \ref{lem:9-isog-sha}.  First, we note that the explicit parametrization provided by Lemma \ref{lem:9-isogeny-param} enables us to construct many elliptic curves for which $\mathcal{P}_3$ is empty by taking $v = 1$.  Thus, every $s \in F^*/F^{*2}$ will be subject to Lemma \ref{lem:9-isog-sha}, so if both $\mathcal{P}_1$ and $\mathcal{P}_2$ are large, then a positive proportion of twists will have large $3$-torsion in $\Sha$.  However, to ensure that this positive proportion is large, we need the rank bound coming from the second isogeny $\phi_2$ to be very efficient.  This bound is most efficient when $|\mathrm{ord}_3 c(\phi_{2,s})| \approx |\omega_1^\mathrm{sp}-\omega_2^\mathrm{sp}|$ is small.  The elliptic curves we construct, therefore, will be chosen so that the sets $\mathcal{P}_1$ and $\mathcal{P}_2$ are close in size.  The following lemma guarantees that we are able to find such curves.

\begin{lemma}\label{lem:rigged-curves}
Let $F$ be a number field.  There exists a constant $r$, depending only of $F$, such that for any $m\geq 1$ there are elliptic curves $E$ with cyclic $9$-isogenies over $F$ with $\mathcal{P}_3$ empty and $\#\mathcal{P}_1,\#\mathcal{P}_2 \in \{m,m+1,\dots,m+r\}$.
\end{lemma}
\begin{proof}
Suppose first that $F=\Q$ and recall the curve $E_{u,v}$ from Lemma \ref{lem:9-isogeny-param}.  By taking $v=1$ and $u\in \mathbb{Z}$, we may guarantee that $\mathcal{P}_3$ is trivial.  In this case, we find that
\[
\mathcal{P}_1 = \{p : p \mid u-1\} \quad \text{and} \quad \mathcal{P}_2 = \{p : p \mid u^2+u+1\}.
\]
Using a standard lower-bound sieve (e.g., the beta-sieve \cite[Theorem 11.13]{FriedlanderIwaniec} of dimension $\kappa=2$), it follows that there are infinitely many $u$ such that neither $u-1$ nor $u^2+u+1$ is divisible by a prime $p \leq |u|^{1/4.84}$.  This implies, in particular, that $u-1$ is divisible by at most $4$ primes while $u^2+u+1$ is divisible by at most $9$ primes.  The same conclusion holds for $u$ satisfying any finite set of congruence conditions, apart from any divisibility conditions imposed by these conditions (e.g., if $u$ is required to be $1 \pmod{p}$, then there are infintely many $u$ such that $u-1$ is divisible by $p$ and at most four other primes).  To prove the lemma when $F=\Q$, we therefore impose congruence conditions on $u$ to guarantee that $u-1$ and $u^2+u+1$ are divisible by at least $m$ primes, and then appeal to this lower-bound sieve.  This is straightforward: let $q_1,\dots,q_{m-1} \equiv 1 \pmod{3}$ be fixed primes, set $q=q_1\dots q_{m-1}$, and let $\alpha$ be such that $\alpha^2+\alpha+1 \equiv 0 \pmod{q}$.  Let $p_1,\dots,p_{m-1}$ be prime, distinct from the $q_i$, and consider $u = p_1\dots p_{m-1} u_1$ where $u_1 \equiv \alpha / (p_1\dots p_{m-1}) \pmod{q}$.  This construction yields the lemma in the case $F=\Q$ with $r=8$.

For general $F$, we again apply a lower-bound sieve, this time to points of bounded height in Minkowski space.  The sieve has exponent of distribution at least $1/d-\epsilon$ for any $\epsilon>0$ and is of dimension $\kappa=2$ if $\mu_3 \not\subset F$ and of dimension $\kappa=3$ if $\mu_3 \subset F$.  Analogous to the case $F=\mathbb{Q}$, we find there is some constant $r$ such that for infinitely many $u \in \mathcal{O}_F$, neither $u-1$ nor $u^2+u+1$ has more than $r$ prime factors; when $\kappa=2$, we may take $r=\lfloor 9.68d\rfloor$, and when $\kappa=3$, we may take $r=2\lfloor 20d/3\rfloor$. By imposing finitely many congruence conditions, the lemma follows.
\end{proof}

We are now ready to make explicit the proof of Theorem \ref{thm:rigged-sha} outlined above.

\begin{proof}[Proof of Theorem $\ref{thm:rigged-sha}$]
We will show that for a curve $E$ constructed in Lemma \ref{lem:rigged-curves}, a proportion at least $1-O(m^{-1/8})$ of twists $E_s$ have $\mathrm{dim}_{\mathbb{F}_3} \Sha(E_s)[3] \geq m-2m^{7/8}$.  Upon taking $m$ sufficiently large, the claim will follow.  Thus, let $m$ be large and let $E$ be a curve constructed as in Lemma \ref{lem:rigged-curves}.  For convenience, assume each $\p$ of multiplicative reduction has norm at least $m$.  

With the notation of Lemma \ref{lem:9-isog-sha}, by varying over $s \in F^*/F^{*2}$, we may think of $\omega_1^\mathrm{sp}$ and $\omega_2^\mathrm{sp}$ as sums of independent Bernoulli random variables.  In particular, at a given prime $\p$ of (potentially) multiplicative reduction, the reduction of $E_s$ is split for a proportion $\frac{1}{2} - \frac{1}{2\mathbf{N}(\p)+2}$ of $s \in F^*/F^{*2}$.  Thus, we find the expected value of $\omega_1^\mathrm{sp}$ to be
\[
\mathbb{E}[\omega_1^\mathrm{sp}] = \sum_{\p \in \mathcal{P}_1} \frac{1}{2}-\frac{1}{2\mathbf{N}(\p)+2} = m/2 + O(1).
\]
A similar computation reveals its variance to be $m/4 + O(1)$, with exactly the same results holding for $\omega_2^\mathrm{sp}$.  We therefore find $\mathrm{ord}_3c(\phi_{1,s}) = -\omega_1^\mathrm{sp}-\omega_2^\mathrm{sp}$ to have expected value $-m + O(1)$ with variance $m/2 + O(1)$.  By Chebyshev's inequality, it follows that a proportion at least $1 - O(m^{-3/4})$ of $s$ are such that $\mathrm{ord}_3 c(\phi_{1,s}) \leq -m + m^{7/8}$.

Similarly, as $\mathrm{ord}_3c(\phi_{2,s}) = \omega_1^\mathrm{sp}-\omega_2^\mathrm{sp}$, we find that it has expected value $O(1)$ with variance $m/2 + O(1)$.  By Chebyshev's inequality again, a proportion at least $1-O(m^{-1/2})$ of twists have $|\mathrm{ord}_3 c(\phi_{2,s})| \leq m^{3/4}$.  For such twists, the average rank of $\mathrm{Sel}_3(E^\prime_s)$ is at most $m^{3/4} + 3^{-m^{3/4}}$, so that at least $1-O(m^{-1/8})$ of these have rank at most $m^{7/8}$.  Pulling this together, we find that a proportion at least $1-O(m^{-1/8})$ of twists $E_s$ have $\dim_{\mathbb{F}_3} \Sha(E_s)[3] \geq m-2m^{7/8}$.  Upon taking $m$ sufficiently large, the result follows.
\end{proof}



\section{Curves with an $18$-isogeny}\label{18-isogeny}

Suppose $E'$ admits a cyclic 18-isogeny.   Then it also possess a cyclic 9-isogeny $\phi\colon E'\to E^{\prime\prime}$, and as before  we decompose $\phi$ as $\phi = \phi_2 \circ \hat\phi_1$, where $\phi_1\colon E \to E^\prime$ and $\phi_2\colon E \to E^{\prime\prime}$ are 3-isogenies.

Using the parametrization in \cite{Maier}, we may replace $E'$, $E$ and $E''$ with appropriate quadratic twists such that $E^\prime$ has an integral model of the form \begin{equation}\label{eq:18isogeny} y^2 - m^3xy = x^3 + (-2n^6 + n^3m^3)x^2 + (n^{12} -  n^9m^3)x\end{equation} with $m,n$ relatively prime. Examining the $c$-invariants of $E^\prime$, we find that the model \eqref{eq:18isogeny} is minimal except possibly at $p = 2$. At $p = 2$, $E^\prime$ will have multiplicative reduction and the model \eqref{eq:18isogeny} will be non-minimal if and only if $m \equiv n \pmod{2}$, in which case $v_2(\Delta^\prime) = v_2(\Delta^\prime_{\mathrm{min}}) + 12$.

The corresponding model for $E^{\prime \prime}$ is then given by \begin{multline}\label{eq:18isogeny2}
y^2 - m^3xy = x^3 + (-6m^5n + 6m^4n^2 - 23m^3n^3 - 12m^2n^4 - 24mn^5 - 2n^6)x^2 \\ - n(m-n)^9(m^2+mn+n^2)x
\end{multline}
The model \eqref{eq:18isogeny2} will be minimal except possibly at $2$ and $3$. It will be non-minimal at $2$ if and only if $m \equiv n \pmod{2}$, in which case $v_2(\Delta^{\prime\prime}) = v_2(\Delta^{\prime\prime}_{\mathrm{min}}) + 12$. It will be non-minimal at $3$ if and only if $m \equiv n \pmod{3}$, in which case $v_3(\Delta^{\prime\prime}) =v_3(\Delta^{\prime\prime}_{\mathrm{min}}) + 24$ and $E^\prime$ has additive reduction at $3$.

These models allow us to easily understand the places where $E$ has bad reduction.

\begin{lemma}
\label{lem:18isog good twist}
If $E$ has additive, potentially good reduction at a prime $p$, then $E$ has a twist $E_s$ with good reduction at $p$. 
\end{lemma}
\begin{proof}
For $p \ne 3$, this is Lemma \ref{lem:9isog good twist}. For $p = 3$, we observe that $v_3(c_4(E^\prime)) = v_3(c_6(E^\prime)) = 0$ if $m \not \equiv n \pmod{3}$ and $v_3(c_4(E^\prime)) = 2$ and $v_3(c_6(E^\prime)) = 3$ if $m \equiv n \pmod{3}$. Since $E$, and therefore $E^\prime$, is assumed to have bad reduction at $3$, we therefore must be in the latter case. Twisting by $3$, we then obtain a curve of good reduction at $3$. 
\end{proof}

\begin{lemma}
\label{lem:18isogtype2primes}
If $p \ge 5$ divides $(n^2+nm+m^2)(4n^2-2nm+m^2)$, then $E$ has multiplicative reduction at $p$ with $v_{p}(j(E^{\prime})) = 3 v_{p}(j(E)) = v_{p}(j(E^{\prime\prime}))$.
\end{lemma}
\begin{proof}
We have $\Delta_{E^\prime} =m^9n^{18}(n-m)^2(2n+m)(n^2+nm+m^2)^2(4n^2-2nm+m^2)$. A resultant computation then shows that any prime $p \ge 5$ dividing $(n^2+nm+m^2)(4n^2-2nm+m^2)$ can't divide $c_4(E^\prime)$, so $E^\prime$ must have multiplicative reduction at $p$. Further, since $E^\prime$ has multiplicative reduction at $p$, we will have $v_p(j(E^\prime)) = -v_p(\Delta_{E^\prime}) = -v_p((n^2+nm+m^2)^2(4n^2-2nm+m^2))$, where the latter equality follows from a resultant computation between $(n^2+nm+m^2)(4n^2-2nm+m^2)$ and each of the other factors of $\Delta_{E^\prime}$.

Since $\Delta_{E^{\prime\prime}} = mn^{2}(n-m)^{18}(2n+m)^9(n^2+nm+m^2)^2(4n^2-2nm+m^2)$, similar considerations show that $v_p(j(E^{\prime\prime})) = -v_p(\Delta_{E^{\prime\prime}}) = -v_p((n^2+nm+m^2)^2(4n^2-2nm+m^2))$. We therefore have $v_p(j(E^\prime)) = v_p(j(E^{\prime\prime}))$, and the result follows from Lemma \ref{lem:9isog mult vals}.
\end{proof}

\begin{corollary}
\label{cor:18isogtype2primes}
If $E$ is a not a twist of a curve in the isogeny class 14a, then there exist distinct primes $\ell_1,\ell_2  \geq 5$ such that $E$ has multiplicative reduction at $\ell_i$ with $v_{\ell_i}(j(E^{\prime})) = 3 v_{\ell_i}(j(E))=v_{\ell_i}(j(E^{\prime\prime}))$.
\end{corollary}
\begin{proof}
It is an elementary exercise to show that the only coprime pairs $(n,m)$ for which there do not exist distinct primes $\ell_1,\ell_2  \geq 5$ dividing $(n^2+nm+m^2)(4n^2-2nm+m^2)$ are $\pm(1,1)$, $\pm(1,-2)$, $\pm(1,-1)$, $\pm(1,2)$, $\pm(2,-1)$, and $\pm(1,4)$. The first two pairs correspond to singular curves and the final four pairs correspond to curves in the isogeny class $14a$ or twists of such curves by $-3$, and the result then follows from Lemma \ref{lem:18isogtype2primes}.
\end{proof}

\begin{lemma}
\label{lem:18isogtype3primes}
If $p \ne 3$ divides the denominator of $\Delta_{E^\prime}/\Delta_{E^{\prime\prime}} = \frac{m^8n^{16}}{(m-n)^{16}(m+2n)^8}$, then $E$ has (potential) multiplicative reduction at $p$ with $9v_{p}(j(E^{\prime})) = 3 v_{p}(j(E))=v_{p}(j(E^{\prime\prime}))$. The same holds for $p = 3$ if it divides  the denominator of $\Delta_{E^\prime}/\Delta_{E^{\prime\prime}}$ to order greater than $24$.
\end{lemma}
\begin{proof}
If $p \ne 3$ divides the denominator of $\Delta_{E^\prime}/\Delta_{E^{\prime\prime}}$, then $E$ must have bad reduction at $p$. The same holds for $p = 3$ if it divides the denominator of $\Delta_{E^\prime}/\Delta_{E^{\prime\prime}}$ to order greater than $24$. For $p \ge 5$, taking the resultant of each of $(m-n)$ and $(m+2n)$ with $c_4(E^{\prime})$ shows that $E^\prime$ must have multiplicative reduction at $p$. For $p = 2$, it suffices, as noted above, that $E^{\prime}$ always has multiplicative reduction at $2$.

For $p = 3$, we will have $p$ dividing the denominator of $\Delta_{E^\prime}/\Delta_{E^{\prime\prime}}$ to order greater than $24$ if and only if $m \equiv n \pmod{9}$ or $m \equiv -2n \pmod{9}$. In each of these cases, we will have $v_3(\Delta_{E^{\prime\prime}}) \ge 32$ and $v_3(\Delta_{E_{\mathrm{min}}^{\prime\prime}}) \ge 8$. As a result, $E^{\prime\prime}$ can't have a twist of good reduction, since that would require $v_3(\Delta_{E_{\mathrm{min}}^{\prime\prime}}) = 6$. Applying Lemma \ref{lem:18isog good twist}, we find that $E$ must have potentially multiplicative reduction.

Finally, we observe that by Lemma \ref{lem:9isog mult vals}, for $p \ne 3$, we will have $9v_{p}(j(E^{\prime})) = 3 v_{p}(j(E))=v_{p}(j(E^{\prime\prime}))$ if and only if $v_p(\Delta_{E^\prime}) < v_p(\Delta_{E^{\prime\prime}})$ and for $p = 3$, we will have $9v_{p}(j(E^{\prime})) = 3 v_{p}(j(E))=v_{p}(j(E^{\prime\prime}))$ if and only if $v_p(\Delta_{E^\prime}) < v_p(\Delta_{E^{\prime\prime}}) - 24$, since the valuation of $p$ in the denominator of the $j$-invariant will be the same as the valuation of the $p$ in the minimal discriminant.
\end{proof}

\begin{corollary}
\label{cor:18isogtype3primes}
If $E$ is not a twist of a curve in the isogeny class 14a, then there exists a prime $\ell_3$ such that $E$ has (potential) multiplicative reduction at $\ell_3$ with $9v_{\ell_i}(j(E^{\prime})) = 3 v_{\ell_i}(j(E))=v_{\ell_i}(j(E^{\prime\prime}))$.
\end{corollary}
\begin{proof}
By Lemma \ref{lem:18isogtype3primes}, it suffices to show that the denominator of $\Delta_{E^\prime}/\Delta_{E^{\prime\prime}} = \frac{m^8n^{16}}{(m-n)^{16}(m+2n)^8}$ is not equal to $\pm 1$ or $\pm 3^{24}$. Elementary arguments show that this is the case for $(n,m)~\ne~\pm~(-2,1)$, $\pm (1,4)$, which correspond to curves in the isogeny class $14a$ or twists~of~such~curves~by~$-3$.\end{proof}

As a consequence of these results, we obtain the following:

\begin{proposition}\label{lem:18-isogeny}
If $E$ is not a quadratic twist of a curve in the isogeny class $14a$, then $T_{-3}(\phi_1) \cap T_{-1}(\phi_2)$ is non-empty.
\end{proposition}

\begin{proof}
By Corollary \ref{cor:18isogtype2primes}, we may find primes $\ell_1$ and $\ell_2$ that fall into case (ii) of Lemma \ref{lem:9isog mult vals}.  Thus, by Lemma \ref{lem:split-mult}, there is a twist $s$ such that $c_{\ell_j}(\phi_{i,s}) = 1/3$ for all $i,j  \in \{1,2\}$.

By Corollary \ref{cor:18isogtype3primes}, we may also find a prime $\ell_3$ such that $v_{\ell_3}(j(E^{\prime\prime})) = 3 v_{\ell_3}(j(E)) = 9 v_{\ell_3}(j(E^\prime))$. 

If $\ell_3 \ne 3$, then $c_{\ell_3}(\phi_{1,s}) = 1/3$ and $c_{\ell_3}(\phi_{2,s}) = 3$ for some $s$ by Lemma  \ref{lem:split-mult}. By Lemma \ref{lem:two-isogenies-trivial} combined with Lemma \ref{prop:local-selmer-comp}, we may additionally find $s$ such that $c_3(\phi_{1,s})c_\infty(\phi_{1,s}) = c_3(\phi_{2,s})c_\infty(\phi_{2,s}) = 1$. Taking $s$ satisfying all of the above further satisfying $c_p(\phi_{1,s}) = 1 = c_p(\phi_{2,s})$ for all $p \nmid 3\ell_1\ell_2\ell_3\infty$, we will then have $c(\phi_{1,s}) = 1/27$ and $c(\phi_{2,s}) = 1/3$,  showing that $T_{-3}(\phi_1) \cap T_{-1}(\phi_2)$ is non-empty.

If $\ell_3 = 3$, then by Table 1 in \cite{DD}, there is some $s$ such that $c_3(\phi_{1,s}) = 1$ and $c_\infty(\phi_{2,s}) = 3$. By Lemma \ref{prop:local-selmer-comp}, we may therefore find $s$ such that $c_3(\phi_{1,s})c_\infty(\phi_{1,s}) = 1/3$ and $c_3(\phi_{2,s})c_\infty(\phi_{2,s}) = 3$. The result then follows as before by taking $s$ satisfying all of the above further satisfying $c_p(\phi_{1,s}) = 1 = c_p(\phi_{2,s})$ for all $p \nmid 3\ell_1\ell_2\ell_3\infty$.
\end{proof}

\subsection{Proof of Theorem \ref{thm:18-isogeny}}

To prove Theorem \ref{thm:18-isogeny}, we need the following result concerning 2-Selmer groups of elliptic curves with rational two-torsion. 

\begin{theorem}\label{thm:erdoskac}
Let $E/\mathbb{Q}$ be an elliptic curve such that $E(\mathbb{Q})[2] \simeq \mathbb{Z}/2\mathbb{Z}$ and let $E^\prime$ be the curve which is $2$-isogenous to $E$.  Suppose that $r_2 \geq 0$ and $\gamma$ is a specified class in $\prod_{v\mid 2N_E \infty} \mathbb{Q}_v^*/\mathbb{Q}_v^{*2}$.  If $\Delta_E \Delta_{E^\prime}$ is not a square, then at least $1/2$ of the twists of $E$ by $s\in \gamma$ have $|\mathrm{Sel}_2(E_s)| \geq 2^{r_2}$.
\end{theorem}
\begin{proof}
This is essentially due to Xiong \cite{xiong} and Klagsbrun and Lemke Oliver \cite{KlagsbrunLO}. While neither result explicitly allows for restricting to $s \in \gamma$, the arguments in both cases can easily be extended to allow for this.
\end{proof}

The following lemma shows that the hypothesis of Theorem \ref{thm:erdoskac} that $\Delta_E \Delta_{E^\prime}$ is not a square is always satisfied in the cases in which we wish to apply the result.

\begin{lemma}\label{lem:notsquare}
Suppose that $E/\mathbb{Q}$ is an elliptic curve such that $E(\mathbb{Q})[2] \simeq \mathbb{Z}/2\mathbb{Z}$ and additionally suppose that $E$ has a rational $3$-isogeny.  If $E^\prime/\mathbb{Q}$ is the curve that is $2$-isogenous to $E$, then $\Delta_E \Delta_{E^\prime}$ is not a square.
\end{lemma}
\begin{proof}
As the modular curve $X_0(6)$ has genus 0, it follows from a rational parametrization given by Maier \cite{Maier} that there is some $t\in \mathbb{Q}$ such that
\[
j(E) = \frac{(t+6)^3(t^3 + 18t^2 + 84t + 24)^3}{t(t+8)^3(t+9)^2}, \quad t\neq 0,-8,-9.
\]
Thus, $E$ is a quadratic twist of a curve with discriminant $t(t+8)^3(t+9)^2 \in t(t+8) \mathbb{Q}^{*2}$.  Since taking quadratic twists changes the discriminant by sixth powers, we find that $\Delta_E \in t(t+8)\mathbb{Q}^{*2}$.  A similar computation reveals that $\Delta_{E^\prime} \in (t+9) \mathbb{Q}^{*2}$.  Thus, $\Delta_E \Delta_{E^\prime}$ is a square if and only if there is a rational $y\neq 0$ such that $y^2 = t(t+8)(t+9)$.  This equation defines an elliptic curve, which is observed to have rank 0 and Mordell-Weil group $\mathbb{Z}/2\mathbb{Z} \times \mathbb{Z}/2\mathbb{Z}$ over $\mathbb{Q}$.  The three points of order two correspond exactly to the trivial solutions $t=0,-8,-9$ ruled out above, and the lemma follows.
\end{proof}

\begin{proof}[Proof of Theorem $\ref{thm:18-isogeny}$]
Let $E'/\mathbb{Q}$ have a cyclic 18-isogeny, and recall we are trying to show that a positive proportion of twists of $E'$ have an element of order 6 in their Tate-Shafarevich groups.  In addition to the $9$-isogeny $\phi$ discussed above, it also follows that $E'$ has a rational two-torsion point.  Moreover, by Lemma \ref{lem:notsquare} $E'$ is subject to Theorem \ref{thm:erdoskac}.  

Suppose that $E'$ is not a twist of a curve in the isogeny class 14a.  By Proposition \ref{lem:18-isogeny}, the set $T_{-3}(\phi_1) \cap T_{-1}(\phi_2)$ is non-empty, so that by Proposition \ref{prop:sha-elements}, a positive proportion of $s\in T_{-3}(\phi_1) \cap T_{-1}(\phi_2)$ are such that $|\Sha(E_s')[3]| \geq 9$.  In fact, the proof of Proposition \ref{prop:sha-elements} shows that this positive proportion may be taken to be $5/6$ in this case.  Combining this with Theorem \ref{thm:erdoskac}, we find that a proportion at least $1/3$ of $s\in T_{-3}(\phi_1) \cap T_{-1}(\phi_2)$ are such that $\Sha(E_s')$ has an element of order 6.  This establishes the theorem.
\end{proof}

Theorem \ref{thm:erdoskac} and Lemma \ref{lem:notsquare} together also quickly imply the following result.

\begin{corollary}
Suppose that $E/\mathbb{Q}$ has a rational degree $6$ isogeny.  Then for any $r_2 \geq 1$ and $\epsilon>0$, $|\Sha(E_s)[2]| \geq 4^{r_2}$ for a proportion at least $1/2 - \epsilon$ of twists $E_s$.
\end{corollary}
\begin{proof}
Since $E$ has a degree $6$ isogeny, it also possesses a degree $3$ isogeny $\phi$.  Let $m$ be any integer for which $T_m(\phi)$ is non-empty.  By Theorem \ref{thm:selmer-average}, it follows that as $r_3 \to \infty$, a proportion $1-O(1/r_3)$ of twists by $s \in T_m(\phi)$ have $|\mathrm{Sel}_3(E_s)| \leq 3^{r_3}$.  Combining this with Theorem \ref{thm:erdoskac} and Proposition \ref{prop:sha-elements}, the claimed result follows for the relative proportion of such $s \in T_m(\phi)$.  Adding these proportions across those $m$ for which $T_m(\phi)$ is non-empty, we obtain the corollary.
\end{proof}



\section{Exploiting Modular Curves}\label{sec:modular-curves}

In this section, we prove Theorem \ref{thm:modular-curve} and Corollaries \ref{cor:p-sha}-\ref{cor:15-21}
We begin with Theorem \ref{thm:modular-curve}.   

\begin{proof}[Proof of Theorem \ref{thm:modular-curve}]
We begin by recalling the setup.  We are given a non-cuspidal point $x\in X_0(3p)(F)$.  We have labeled the cusps $\mathfrak{c}_i$, $i \in \{1,3,p,3p\}$, of $X_0(3p)$ according to their ramification degree, and we have set $\omega_i$ for the number of primes at which $x$ and $\mathfrak{c}_i$ have the same reduction.  
Now, let $E/F$ be any elliptic curve corresponding to the point $x$ on $X_0(3p)(F)$.  Then $E$ has an $F$-rational $3$-isogeny $\phi\colon E \to E^\prime$ and an $F$-rational $p$-isogeny $\psi\colon E \to E^{\prime\prime}$.  We will ultimately apply Proposition \ref{prop:sha-elements} to these two isogenies, so we begin by analyzing their Selmer ratios.

Let $\p \nmid 3p$ be a prime for which $\bar{x} = \bar{\mathfrak{c}_i}$ for some $i$.  If $r \geq 1$ is such that $x$ and $\mathfrak{c}_i$ have the same reduction $\pmod{\p^r}$ but not $\pmod{\p^{r+1}}$, then $v_\p(j(E)) = -ri$.  Moreover, by considering the action of the Fricke involutions $W_3$ and $W_p$ on $X_0(3p)$, we find that
\[
(-v_\p(j(E^\prime)), -v_\p(j(E^{\prime\prime}))) = \left\{\begin{array}{ll} (3r,pr) & \text{if } i=1, \\ (r,3pr) & \text{if } i=3, \\ (3pr,r) & \text{if } i=p, \text{ and}\\ (pr,3r) & \text{if } i=3p. \end{array}\right.
\]
Taking $E_s$ to be a twist such that $E_s$ has split multiplicative reduction at $\p$, we get  $c_\p(\phi_s) = 3$ and $c_\p(\psi_s) = p$ in the case $i = 1$ and $c_\p(\phi_s) = 1/3$ and $c_\p(\psi_s) = p$ in the case $i = 3$ by Lemma \ref{lem:split-mult}.  Thus,
\[
\prod_{\p: \bar{x} = \bar{\mathfrak{c}}_1 \text{ or } \bar{\mathfrak{c}}_3} c_\p(\phi_s) = 3^{\omega_1 - \omega_3} \text{  and  } \prod_{\p: \bar{x} = \bar{\mathfrak{c}}_1 \text{ or } \bar{\mathfrak{c}}_3} c_\p(\psi_s) = p^{\omega_1 + \omega_3}.
\]
At all other primes $\p \nmid 3p$, we may choose $s$ so that $c_\p(\phi_s) = c_\p(\psi_s)=1$.  At primes $\p\mid 3$, we choose $s$ so that $c_\p(\psi_s) = 1$.  At worst, for this $s$ we have $$3^{-d} \leq \prod_{\p\mid 3} c_\p(\phi_s) \leq 3^{2d}.$$  At $\p\mid p$, we may at the very least choose $s$ so that $c_\p(\psi_s) \geq 1$ while maintaining $c_\p(\phi_s) =1$.  Compiling these contributions, we have
\[
3^{\omega_1-\omega_3-d} \leq \prod_{\p < \infty} c_\p(\phi_s) \leq 3^{\omega_1 - \omega_3 + 2d} \quad \text{and} \quad \prod_{\p < \infty} c_\p(\psi_s) \geq p^{\omega_1+\omega_3}.
\]
Thus, there are two extremes to be concerned with: either $v_3(c(\phi_s))$ could be large and positive, or $v_3(c(\phi_s))$ could be large and negative.  Considering the infinite places, in the first case, there is a choice of $s$ for which $v_p(c(\psi_s)) - v_3(c(\phi_s)) \geq 2\omega_3 - 2d$, while in the latter, there is a choice for which $v_p(c(\psi_s))-|v_3(c(\phi_s))| \geq 2\omega_1 - 2d$.  The result now follows from Proposition \ref{prop:sha-elements}.
\end{proof}

We now proceed to the proofs of the associated corollaries to Theorem \ref{thm:modular-curve}.

\begin{proof}[Proof of Corollary \ref{cor:p-sha}]
Let $x$ be the point on the modular curve $X(1)$ corresponding to $E$ and let $F$ be the field $\mathbb{Q}(x^\prime)$ where $x^\prime$ is a preimage of $x$ under the degree $4p+4$ covering $X_0(3p) \to X(1)$.  At each prime $\ell$ at which $E$ has multiplicative reduction, the point $x$ reduces to the (unique) cusp of $X(1)$, so at any prime $\mathfrak{l}$ of $F$ lying over $\ell$, the point $x^\prime$ must reduce to one of the four cusps of $X_0(3p)$.  In fact, by our assumption on the Galois action on $E[3p]$, we must have that $\ell \mathcal{O}_F = \mathfrak{l}_1\mathfrak{l}_3\mathfrak{l}_p\mathfrak{l}_{3p}$, where the reduction of $x^\prime$ on $X_0(3p) \pmod{\mathfrak{l}_i}$ is the same as that of $\mathfrak{c}_i$.  As we have assumed that there are at least $4p+4+r$ such primes $\ell$, it follows that each $\omega_i \geq 4p+4+r$.  The result now follows from Theorem \ref{thm:modular-curve}.
\end{proof}

\begin{proof}[Proof of Corollary \ref{cor:5-sha}]
Let $E\colon y^2 + xy + y = x^3 + x^2 - 13x - 219$ be the elliptic curve with Cremona label 50b3.  We wish to show that $|\Sha(E_s)[5]|\geq 25$ for an explicit set of $s$.  $E$ has both a 3-isogeny $\phi\colon E \to E^\prime$ and a 5-isogeny $\psi\colon E \to E^{\prime\prime}$, where $E^\prime$ and $E^{\prime\prime}$ have Cremona labels 50b4 and 50b1, respectively.  We claim that if $s \equiv 1 \pmod{8}$ is a positive squarefree integer coprime to $5$, then the two global Selmer ratios are given by $c(\phi_s) = 1$ and $c(\psi_s) = 25$.  In particular, for all but finitely many such $s$, we will have that $|\mathrm{Sel}_5(E_s)| \geq 25$, and for $50\%$ of such $s$ the rank of $E_s(\mathbb{Q})$ will be 0 by Theorem \ref{thm:selmer-average}.  Thus, the theorem will follow from the claim about the global Selmer ratios of $\phi_s$ and $\psi_s$.

The curve $E$ has split multiplicative reduction of Kodaira type $\In{1}$ at $p=2$ and additive reduction of Kodaira type $\IIS$ at $p=5$.  By Proposition \ref{prop:local-selmer-comp}, it follows that $c_2(\phi_s) = 3$ for all squarefree $s \equiv 1 \pmod{8}$ and $c_5(\phi_s)=1$ for all $s$ that are coprime to $5$.  In addition, a computation in \verb^Magma^ shows that 
$c_3(\phi_s) =1$ for all $s$ and $c_\infty(\phi_s) = 1/3$ for all positive $s$.  We thus find that for $s$ as claimed, we have $c(\phi_s) = 1$.

We now consider $c(\psi_s)$.  From \cite[Table 1]{DD}, we see that $c_2(\psi_s) = 5$ for all squarefree $s \equiv 1 \pmod{8}$.  The field $\mathbb{Q}(\ker \psi) = \mathbb{Q}\Big(\sqrt{5\sqrt{5}-50}\Big)$ is totally complex, so it follows that $c_\infty(\psi_s) = 1$ for all positive $s$.  Lastly, $c_5(\psi_s) = c_5(E^{\prime\prime}_s)/c_5(E_s) \cdot \alpha_{\psi_s,\mathbb{Q}_5}$.  Since $E$ has Kodaira type $\IIS$ at $p=5$, it follows that $c_5(E^{\prime\prime}_s)/c_5(E_s) = 1$ for all $s$.  Observe that $\Delta_E = -2\cdot 5^{10}$ while $\Delta_{E^{\prime\prime}} = - 2^5 \cdot 5^2$, so that by \cite[Theorem 1]{kg}, $\alpha_{\psi_s,\mathbb{Q}_5} = 5$ for all $s$ coprime to $5$.  Pulling this together, we find that $c(\psi_s) = 25$ for all positive squarefree $s \equiv 1 \pmod{8}$ that are coprime to $5$, and the theorem follows.  The claim about the elliptic curve 50b4 follows along the same lines.
\end{proof}

We now turn to the proof of Corollary \ref{cor:15-21} concerning fields over which the modular curves $X_0(15)$ and $X_0(21)$ have infinitely many points.  Recall that $X_0(15)$ and $X_0(21)$ both have genus one, so that they may be given the structure of an elliptic curve.  The following lemma will be used to find rational points which reduce to specified cusps modulo many primes.

\begin{lemma}\label{lem:ec-divis}
Let $E/F$ be an elliptic curve of positive rank and let $T \in E(F)$ be a non-trivial torsion point.  Fix an integral model for $E$.  Given any $\omega_1$ and $\omega_3$, there exist distinct primes $\p_1,\dots,\p_{\omega_1}$ and $\q_1,\dots,\q_{\omega_3}$ for which there are infinitely many points $P \in E(F)$ for which $P \equiv \mathcal{O} \pmod{\p_i}$ for each $i\leq \omega_1$ and $P \equiv \bar{T} \pmod{\q_j}$ for each $j \leq \omega_3$.
\end{lemma}
\begin{proof}
Let $P \in E(F)$ be of infinite order.  Let $d$ be the order of $T$.  Let $w = \max(\omega_1,\omega_3)$, and let $\ell_1, \dots, \ell_w \in \mathbb{Z}$ be any odd primes congruent to $1\pmod{d}$ and sufficiently large that both $\ell_i P$ and $\ell_i P - T$ have a denominator divisible by a prime not dividing the denominator of $P$ or $P-T$; as any elliptic curve has only finitely many $S$-integral points (see Corollary IX.3.2.1 in \cite{AEC}, for example), this is always possible.  Let $\p_i$ be a prime for which the denominator of $\ell_i P$ has a non-trivial valuation and let $\q_i$ be such a prime for $\ell_i P - T$.  Set $\ell = \ell_1 \dots \ell_w$.  Then for any integer $n\equiv 1 \pmod{d}$, the point $n\ell P$ satisfies the desired conditions. 
\end{proof}

\begin{proof}[Proof of Corollary \ref{cor:15-21}]
Suppose that $p=5$ or $p=7$.  The embedding $X_0(3p) \to J_0(3p)$ given by $x \mapsto [x] - [\mathfrak{c}_1]$ is an isomorphism, endowing $X_0(3p)$ with the structure of an elliptic curve $E$.  Moreover, by the Manin-Drinfeld theorem, the cusps $\mathfrak{c}_i$ of $X_0(3p)$ are torsion points in the Mordell-Weil group $E(F)$.  For any $r \geq 1$, let $\omega_1 = \omega_3 = r + 2d$.  Applying Lemma \ref{lem:ec-divis} with $T = [\c_3]-[\c_1]$, we find infinitely many points $x \in X_0(3p)(F)$ for which Theorem \ref{thm:modular-curve} produces a curve $E/F$ with $|\Sha(E_s/F)[p]| \geq p^{2r}$ for a positive proportion of twists.  This is Corollary \ref{cor:15-21}.
\end{proof}

\begin{remark}\label{rem:100}{\em
The proof of Lemma \ref{lem:ec-divis} could likely be adapted to show that when the points $P$ of $E(F)$ are ordered by height, almost all will be such that the conclusion of the lemma holds for fixed values of $\omega_1$ and $\omega_3$.  For example, most integers $n$ have at least $\frac{1}{2}\log\log n$ prime factors, so that there are at least $\frac{1}{2} \log\log n$ primes contributing to $\omega_1$ for most points $nP$.  Similarly, most $n$ also have at least $\frac{1}{2\phi(d)} \log\log n$ prime factors congruent to $1 \pmod{d}$.  Since $\ell nP - T = \ell\cdot(nP - T)$ for such a prime $\ell$, it follows that most points $nP$ will also have $\geq \frac{1}{2\phi(d)} \log\log n$ primes contributing to $\omega_3$.  This argument essentially suffices in the case that $E(F)$ has rank $1$, and we expect an analogous argument can be made in higher rank.}
\end{remark}

\section{Exploiting primes of (potential) good reduction}\label{arb p}

Let $p$ be an odd prime.  If $E/\Q$ is an elliptic curve with irreducible $\Gal(\bar \Q/\Q)$-representation $E[3p]$, then our techniques do not say anything about the average size of $\Sha(E_s)[p]$ as $s$ varies over $\Q^\times/\Q^{\times2}$, since $E$ admits neither a 3-isogeny nor a $p$-isogeny. 

This is of course no longer the case if we base change $E$ to a sufficiently large extension, a fact that we took advantage of in the proof of Corollary \ref{cor:p-sha}. However, Corollary \ref{cor:p-sha} requires that $E$ have a large number of places of multiplicative reduction, imposing a significant restriction on $E$.

In the proof of Theorem \ref{large primes} that follows, we show how similar results can be obtained by exploiting primes dividing the degrees of the two isogenies. This allows us to extend our results to many additional curves, including those with everywhere potentially good reduction.

\begin{proof}[Proof of Theorem $\ref{large primes}$]
We recall the hypotheses of the theorem, that $E/\mathbb{Q}$ is an elliptic curve with potentially good and ordinary reduction at both $3$ and $p$ and that $\mathrm{Gal}(\bar{\mathbb{Q}}/\mathbb{Q})$ acts transitively on the set of $\mathbb{F}_\ell$ lines in $E[\ell]$ for both $\ell =3$ and $\ell = p$. It follows that $\Gal(K/\Q)$ acts transitively on the $\mathbb{F}_\ell$ lines in $E_s[\ell]$ as well.

We now replace $E$ by its base change to $K = \mathbb{Q}(E[3p])$, at which point it has everywhere semi-stable reduction. Let $\psi \colon E \to E'$ be any of the $p + 1$ isogenies of degree $p$ emanating from $E$ defined over $K$, and $\phi \colon E \to E_0$ be any of the four 3-isogenies out of $E$, all of which are defined over $K$ as well. We restrict our focus to the subset $S$ of elements $s \in K^*/K^{*2}$ such that \[c_\ell(\phi_s) = 1 = c_\ell(\psi_s)\] for all primes $\ell$ of $K$ at which $E$ has multiplicative reduction. By Lemmas \ref{lem:split-mult} and \ref{lem:twistofgoodred}, this condition holds whenever $E_s$ does not have split multiplicative reduction at $\ell$, so $S$ has positive density. We now claim that the average rank of $E_s$ for $s \in S$ is at most $\frac{d}{2} + 3^{-6}$, where $d = \frac{1}{2}[K:\QQ]$.

Since $c_\ell(\phi_s) = 1$ at all places where $E$ has bad reduction, Lemma \ref{cor:loc-const} tells us that $c(\phi_s) = c_\infty(\phi_s)c_3(\phi_s)$, where $c_p(\phi) = \prod_{\p \mid p} c_\p(\phi)$.  As $K$ is totally complex, we have $c_\infty(\phi_s) = 3^{-d}$.  On the other hand, if $\p \mid 3$, then $c_\p(\phi_s)$ is either $3^{[K_\p \colon \Q_3]}$ or 1, depending on whether $\ker\phi_s$ reduces mod $\p$ to the kernel of Frobenius or not, or in other words, whether $\ker \phi_s$ is the canonical subgroup of $E[3]$ over $K_\p$.  The latter condition is equivalent to saying that $\ker\phi_s$ lies in the formal group $\hat E_s$ over $K_\p$.

The primes of $K$ above $3$ are permuted transitively by $\Gal(K/\Q)$ and this Galois action is compatible the $\Gal(K/\Q)$-action on the canonical subgroups: if $\sigma \in \Gal(K/\Q)$ then the canonical subgroup of $E$ over $K_{\p^\sigma}$ is $(\ker \phi)^\sigma$.  It follows that $\ker \phi_s$ is the kernel of Frobenius for exactly 1/4 of all primes $\p$ of $K$ above 3.  Therefore $c_3(\phi_s) = 3^{d/2} \cdot 1^{3d/2} = 3^{d/2}$, which gives $c(\phi_s) = 3^{-d/2}$.  Hence, by Theorem \ref{thm:selmer-average}, the average rank of $E_s$ for $s \in S$ is at most $\frac{d}{2} + \frac{1}{3^{d/2}}.$  Since $K$ contains $\zeta_3$ and $\zeta_p$, we have $2(p-1) \mid 2d$.  Thus, $d/2$ is an integer greater than 1, and the average rank of $E_s$, for $s \in S$, is at most $\frac{d}{2} + \frac19$. 

Turning our attention to $\psi$, we observe that the same reasoning yields $c(\psi_s) = p^{2d/(p+1) - d}$.  It follows that $\Sel_p(E'_s)$ has $\F_p$-dimension at least $d - \frac{2d}{p+1}$ for all $s \in S$.  However, a positive proportion of twists $E'_s$ by $s \in S$ have rank at most $\frac{d}{2}$.  For these $s$, we conclude that $\dim_{\F_p} \Sha(E'_s)[p]$ is at least
\begin{equation}
\label{eq:dimSha}
d - \frac{2d}{p+1} - \frac{d}{2} = \frac{d}{2}\left(1 - \frac{4}{p+1}\right),
\end{equation}
which tends to $\frac{d}{2}$ as $p \to \infty$ and is positive for all $p \ge 5$.
\end{proof}


\section{Tate-Shafarevich groups of CM elliptic curves}\label{cm}

In this section we prove lower bounds for the average order of $\Sha(E_s)[n]$, for a large class of elliptic curves $E$ with complex multiplication, which is the content of Theorem \ref{CMsha}.  We will need two preliminary results.  

\begin{lemma}\label{cmisogeny}
Let $E$ be an elliptic curve  over a number field $F$, and suppose $\End_F(E)$ is the quadratic ring of discriminant $Df^2$, with $D$ a fundamental discriminant.  Then there exists a cyclic $f$-isogeny $\phi \colon E \to E_0$ such that $\End_F(E_0)$ has discriminant $D$.  
\end{lemma}
\begin{proof}
Let $K = \Q(\sqrt{D})$, and let $\O_K$ be its ring of integers.  We identify $\End_F(E)$ with the order $\O$ of index $f$ inside $\O_K$.  Then $f\O_K \subset \O$ is an $\O$-ideal and the desired isogeny $\phi$ is the isogeny $\phi \colon E \to E/E[f\O_K]$.  Indeed, $\ker \phi \simeq E[f\O_K] \simeq \O/f\O_K \simeq \Z/f\Z$, so $\phi$ is a cyclic $f$-isogeny.  Moreover, if we define $E_0 = E/E[f\O_K]$, then $\End_F(E_0) \simeq \O_K$, by \cite[Thm.\ 20]{Kani}.  
\end{proof}

\begin{proposition}\label{lem: volcano}
Suppose $\phi \colon E \to E_0$ is as in Lemma $\ref{cmisogeny}$, and $f = p^n$ for some prime $p$.  If $\p$ is a prime of $F$ above $p$ of $($potentially$)$ ordinary reduction for $E$, then $c_\p(\phi) = p^{n[F_\p \colon \Q_p]}$.    
\end{proposition}      
\begin{proof}
Since $\p$ is a prime of potentially ordinary reduction, by \cite[Table 1]{DD}, to compute $c_\p(\phi)$, we may replace $F_\p$ by a finite extension over which $E$ has good ordinary reduction.  So we may assume that this is the case already for $E$ over $F$.

We let $\bar \phi \colon \bar E \to \bar E_0$ be the induced isogeny of elliptic curves over the residue field $\F_\p$.  The key point is that $\ker \bar\phi$ is connected, i.e.\ $\bar\phi$ is (up to isomorphism) the $n$th power of the absolute Frobenius isogeny of $E$ over $\F_\p$.  In other words, $\ker \phi$ is the canonical subgroup of $E[p^n]$.  To see this, note that $\psi \circ \phi = [p^n]$, for some cyclic $p^n$-isogeny $\psi \colon E_0 \to E$.  The canonical subgroup $C$ in $E_0[p]$ is of the form $E_0[\a]$ for some ideal $\a \subset \O_K$, so $\End(E_0/C) \simeq \O_K$.  It follows that $\ker \psi$ intersects trivially with the canonical subgroup of $E_0[p]$.  Indeed, if the intersection were non-trivial, then $\psi \colon E_0 \to E$ would factor through an isogeny of degree $p^{n-1} \colon E_0/C \to E$.  This is impossible, since $\Disc(\End(E)) = p^n\Disc(\End(E_0/C))$ and the discriminant changes by at most a factor of $p$ under a $p$-isogeny (see e.g.\ \cite[Cor.\ 4.3]{RosenShnidman}).

We conclude that $\ker\bar\psi$ is \'etale, and hence $\ker \bar\phi$ is connected.  In other words, $\ker\phi$ reduces to the formal group of $E$.  Using \cite[Table 1]{DD}, we conclude that $c_\p(\phi) = p^{n[F_\p \colon \Q_p]}$.
\end{proof}

\begin{proof}[Proof of Theorem $\ref{CMsha}$]
On the one hand, by \cite[Thm.\ 2.7]{BKLOS}, the average rank of $E_s$ is at most 1.  Since the rank of $E_s(F)$ is even, this means that at least $50\%$ of these twists have rank 0.  On the other hand, we will show that for all but finitely many $s \in F^\times/F^{\times 2}$, the $f$-Selmer group $\Sel_f(E_s)$ has size at least $f^d$.  For those twists with rank 0, this implies that    $|\Sha(E_s)[f]| \geq f^d$, proving the theorem. 

To give a lower bound for $|\Sel_f(E_s)|$, we choose $E_0$ over $F$ with $\End_F(E_0) \simeq \O_K$ and such that there is a cyclic isogeny $\phi \colon E \to E_0$ of degree $f$,  as in Lemma \ref{cmisogeny}.  We will show that $c(\phi_s) = f^d$, for all $s$.  From Proposition \ref{prop:selmer-ratio}, it will then follow that $\Sel_\phi(E_s)$ has size at least $f^d$, and hence $\Sel_f(E_s)$ has size at least $f^d$ for all but finitely many $s$, which will complete the proof.   

To compute $c(\phi_s)$, it suffices to consider the case $s = 1$.  We need to compute $c_\p(\phi)$ for all primes $\p$ of $F$.  If $\p$ is a finite prime not dividing $f$, then Lemma \ref{lem:twistofgoodred} implies that $c_\p(\phi) = 1$ since $E_0$ has a quadratic twist of good reduction (see \cite[Proof of Thm.\ 11.2]{BKLOS}).
Next we consider primes $p$ dividing $f$, and primes $\p$ of $F$ above $p$.  Then $E$ has potentially ordinary reduction at $\p$, since $p$ splits in $K$.  We can factor $\phi \colon E\to E_0$ into a $\phi_2 \circ \phi_1$ with $\phi_2$ a $p$-power isogeny and $\phi_1$ a prime-to-$p$ isogeny.  As noted in Section \ref{sec:global_selmer_ratios}, we then have $c_\p(\phi) = c_\p(\phi_1)c_\p(\phi_2) = c_\p(\phi_2)$ by Lemma 7.2(b) in \cite{milne}.  Applying Proposition \ref{lem: volcano}, this is equal to $p^{n[F_\p \colon \Q_p]}$, where $p^n$ is the highest power of $p$ dividing $f$.    

Finally, if $\p$ is archimedean, then $\p$ is complex since $F$ necessarily contains $K$.  We therefore have $\prod_{\p \mid \infty} c_\p(\phi) = f^{-[F \colon K]}$.  Putting all of the local computations together, we conclude that
\begin{multline}
c(\phi) = \left(\prod_{p \mid f} \prod_{\p \mid p} c_\p(\phi)\right) \prod_{\p \mid \infty} c_\p(\phi)  = \left ( \prod_{p \mid f} \prod_{\p \mid p} p^{n[F_\p \colon \Q_p]} \right) f^{-[F \colon K]}= \left ( \prod_{p \mid f} p^{n[F \colon \Q]} \right) f^{-[F \colon K]}\\ = \left ( \prod_{p \mid f} p^n\right)^{[F \colon \Q]}  f^{-[F \colon K]} = f^{[F \colon \Q]} f^{-[F \colon K]} = f^{[F \colon K]} = f^d
\end{multline}
as desired.              
\end{proof}

Note that the interesting cases of Theorem \ref{CMsha} are for $f \geq 3$, and in those cases $F$ has degree at least 4 over $\Q$.  The degree of such an $F$ necessarily grows with $f$, since the field of definition for any CM elliptic curve with $\End(E)$ of discriminant $Df^2$ is the ring class field of $K$ of conductor $f$.  

\begin{example}
{\em If $f = 3$, we can take $K$ to be any imaginary quadratic field in which 3 splits.  In this case, $F$ must contain $H(\sqrt{-3})$, where $H$ is the Hilbert class field of $K$.  Indeed, $H(\sqrt{-3})$ is the ring class field of $K$ of conductor $3$ whenever $3$ splits in $K$.  If $K$ has class number 1, then we can take $F = K(\sqrt{-3})$, which is biquadratic over $\Q$.  Theorem \ref{CMsha} then says that at least $50\%$ of  twists $E_s$ have rank 0 and satisfy $|\Sha(E_s)[3]| \geq 9$.  If we base change this $E$ to a number field $F'$ of degree $2d$ over $\Q$, then half of all twists over $F'$ have rank 0 and $\Sha(E_s)[3]$ of size at least $3^d$.    
}
\end{example}


\section{Tate-Shafarevich groups of CM abelian varieties}\label{ab vars}


The approach used in the previous section can be extended to more general CM abelian varieties.  We spell out the details in a particularly pretty example. 

Let $J$ be the   Jacobian of the genus three Picard curve $C \colon y^3 = x^4 - x$.  Over $\Q$, $J$ has good reduction away from 3.  Moreover, $J$ is absolutely simple and has CM by $K = \Q(\zeta_9)$; see \cite[\S12]{BKLOS}.  The complex multiplication is defined over all fields containing $K$, so we will work for now over a general number field $F$ containing $K$.  Also write $K^+$ for the maximal totally real cubic subfield of $K$, which is an abelian cubic extension of $\Q$.    

Let $p$ be a prime of ordinary good reduction for $J$ over $\Q$.  For example, we could take $p$ to be any prime which splits completely in $K$, or in other words such that $p \equiv 1 \pmod 9$.  We can then write $p\O_K = \prod_{i = 1}^6 \p_i$, and $J[p] \simeq \bigoplus J[\p_i]$.  Let us now assume that $F$ contains the field $K(J[p])$ over which the action of $G_K$ on the Galois module $J[p]$ of order $p^6$ becomes trivial.    

For any prime $\p$ of $F$ above $p$, we write $J_{\F_\p}$ for the reduction of $J$ over $\F_\p = \O_F/\p$.  Over the completion $F_\p$, there is a unique subgroup $C_\p$ of $J[p]$ of order $p^3$ which lifts the kernel of the absolute Frobenius $J_{\F_p} \to J_{\F_p}^{(p)} \simeq J_{\F_p}$.  In fact, we may write $C_\p = J[\n_\p]$, where $\n_\p$ is a product of three prime ideals above $p$.  In fact, if we order appropriately, the six ideals $\n_\p$ are:
\[\p_1\p_2\p_3,\,  \p_2\p_3\p_4,\, \p_3\p_4\p_5\]
and their complex conjugates
\[\p_4\p_5\p_6,\,  \p_5\p_6\p_1,\, \p_6\p_1\p_2.\] 
Note that $p\O_K = \n_\p \bar\n_\p$ and $A[p] = A[\n_\p]\oplus A[\bar \n_\p]$, where $\bar \n_\p$ denotes the complex conjugate of $\n_\p$.  Also note that $C_\p$ is defined over the number field $K$, and hence $F$ as well.  We refer to the $C_\p$ as {\it canonical} subgroups.    

\begin{definition}{\em  A subgroup $C \subset J[p]$ of order $p^3$ is {\it anti-canonical} if it intersects trivially with all six  canonical subgroups $C_{\p_i}$.}  
\end{definition}
There are many anti-canonical subgroups of $J[p]$.  We describe one such  subgroup below.
\begin{example}
{\em Over $\C$ we have $J \simeq \C^3/\O_K$, where the embedding of $\O_K$ as a full rank lattice in $\C^3$ is via the CM-type of $J$.  We use this embedding to view all lattices in $K$ as lattices in $\C^3$.  Let $\O_p$ be the order $\O_{K^+} + p\O_K$ of index $p^3$ inside $\O_K$.  There are natural $(\Z/p\Z)^3$-isogenies of complex tori
\[\C^3/\O_K \simeq \C^3/p\O_K \to \C^3/\O_p \hspace{4mm} \mbox{and} \hspace{4mm} \C^3/\O_p \to \C^3/\O_K\]
which descend to isogenies of abelian varieties $\psi \colon J \to A$ and $\phi \colon A \to J$ over $F$. The composition $\phi \circ \psi$ is simply multiplication-by-$p$ on $J$.  Note also that $\End_F(A) \simeq \O_p$ and $\ker \phi = A[p\O_K]$.  If we denote the kernel of $\psi$ by $C$, then we claim that $C \subset J[p]$ has trivial intersection with all six canonical subgroups $C_{\p} = J[\n_\p]$.  This follows from the fact that $\O_p \cap \n_\p = p\O_K$.}
\end{example}

For our purposes, anti-canonical subgroups $C \subset J[p]$ are interesting because they reduce faithfully into the $p$-torsion of the reduction $J_{\F_\p}$, for all primes $\p$ of $K$ above $p$.  In particular, the natural $p^3$-isogeny $\psi \colon J \to J/C$ induces an isomorphism on formal groups, and so $\alpha_{\psi, K_\p} = 1$.  This is the last bit of input we need to prove the following theorem.  

\begin{theorem}
Let $p \equiv 1 \pmod 9$ be a prime, and let $F$ be a number field containing $K(J[p])$.  Let $C$ be any anti-canonical subgroup of $J[p]$, and set $A = J/C$.  Then at least $50\%$ of all quadratic twists $A_s$ have rank $0$ and satisfy $|\Sha(A_s)[p]| \geq p^{3d}$, where $[F \colon \Q] = 2d$.
\end{theorem}

\begin{proof}
By \cite[Thm.\ 12.5]{BKLOS}, the average rank of $J_s(F)$ is at most 3, and at least $50\%$ of twists have rank 0.  It follows that at least $50\%$ of twists $A_s$ have rank 0 as well.  Let $\psi \colon J \to A$ be the natural quotient with kernel $C$, and let $\phi \colon A \to J$ be the unique isogeny (of degree $p^3$) such that $\phi \circ \psi = [3]$.  We will show that the isogeny $\phi_s \colon A_s \to J_s$ satisfies $c(\phi_s) = p^{3d}$.  It will then follow that for all but finitely many $s \in F^\times/F^{\times2}$, we have $|\Sel_p(A_s)| \geq p^{3d}$. Hence, for those $A_s$ with rank 0, we must have $|\Sha(A_s)[p]| \geq p^{3d}$, which proves the theorem. 

To compute $c(\phi_s)$, we first argue that $c_\q(\phi_s) = 1$ for any prime $\q \nmid p\infty$.  This follows from Lemma \ref{cor:loc-const} and the fact that $A_s$ and $J_s$ have quadratic twists of good reduction at $\q$.  ($J$ has good reduction at all primes of $F$ above 3 by \cite[\S2]{serre-tate}.)  On the other hand, $c_\infty(\phi_s) = p^{-3d}$, since $F$ is totally complex.  For primes $\p$ of $F$ above $p$, we claim that $c_\p(\phi_s) = p^{3[F_\p \colon \Q_p]}$.  First note that $c_\p(\phi_s) = \alpha_{\phi_s,F_\p}$ since $A_s$ and $J_s$ have twists of good reduction.  
We also have $\alpha_{\phi_s, F_\p} = \alpha_{[p], F_\p}\alpha_{\psi_s, F_\p}^{-1}$.  Since $C_s$ is anti-canonical, the extension of $C_s$ to the N\'eron model of $A$ over $\O_{F_\p}$ is \'etale, and hence
\[\alpha_{\phi, F_\p} = \alpha_{[p], F_\p}\alpha_{\psi, F_\p}^{-1} = p^{3[F_\p \colon \Q_p]} \cdot 1 = p^{3[F_\p \colon \Q_p]}.\]  
Putting everything together, we find that 
\[c(\phi_s) = p^{-3d} \cdot  \left (\prod_{\p \mid p} p^{3[F_\p \colon \Q_p]} \right) = p^{-3d}\left(p^{3[F:\QQ]} \right)= p^{-3d}p^{6d}= p^{3d},\]
which concludes the proof.
\end{proof}


\section{Example}
\label{sec:examples}

This section concerns the isogeny class of elliptic curves with Cremona label 14a.  In particular, the curve with Cremona label 14a1
\[
E\colon y^2 + xy + y = x^3+4x-6
\]
has conductor $14$ and admits two independent $3$-isogenies $\phi_1\colon E \to E^\prime$ and $\phi_2\colon E \to E^{\prime\prime}$, where $E^\prime$ and $E^{\prime\prime}$ have Cremona labels 14a3 and 14a4, respectively.  Over $\mathbb{Q}_2$, the curves $E$, $E^\prime$, and $E^{\prime\prime}$ have Kodaira types $\In{6}$, $\In{18}$, and $\In{2}$, respectively.  Over $\mathbb{Q}_7$, $E$ has type $\In{3}$, while $E^\prime$ and $E^{\prime\prime}$ have type $\In{1}$.  Additionally, we find that $\mathbb{Q}_3(\ker\phi_1)$ is unramified while $\mathbb{Q}_3(\ker\phi_2)$ is ramified.

From Proposition \ref{prop:local-selmer-comp} and a straightforward computation, we see that
\[
\mathbb{Z} \setminus \{0\} = T_{-2}(\phi_1) \cup T_{-1}(\phi_1) \cup T_0(\phi_1) \cup T_1(\phi_1).
\]
To compute the densities of these sets, we find it convenient to use a generating function (in fact, a generating polynomial).  The valuation of $c_\infty(\phi_{1,s})$ is either $0$ or $-1$, each with probability $1/2$, corresponding to the sign of $s$, and we compute that $c_3(\phi_{1,s})=1$ for all $s$.  Thus, we find
\[
\sum_{m\in \mathbb{Z}} \mu(T_m(\phi_1)) x^m = \left(\frac{1+x^{-1}}{2}\right) \left(\sum_{U \in \mathbb{Q}_2^\times/(\mathbb{Q}_2^\times)^2} x^{\mathrm{ord}_3 c_2(\phi_{1,s})}\mu_2(U) \right) \left(\sum_{U \in \mathbb{Q}_7^\times/(\mathbb{Q}_7^\times)^2} x^{\mathrm{ord}_3 c_7(\phi_{1,s})} \mu_7(U)\right),
\]
where $\mu_p(U)$ denotes the $p$-adic Haar measure of the set $U \cap \mathbb{Z}_p$ and $s$ is chosen to be a representative of $U$.  We compute the individual factors and find
\begin{align*}
\sum_{m\in \mathbb{Z}} \mu(T_m(\phi_1)) x^m
	&= \left(\frac{1+x^{-1}}{2}\right) \left(\frac{x}{6} + \frac{5}{6}\right) \left(\frac{27}{48}+\frac{21}{48}x^{-1}\right) \\
	&=\frac{35}{192}x^{-2} + \frac{29}{64}x^{-1} + \frac{61}{192} + \frac{3}{64}x.
\end{align*}
Thus, for example, the set $T_0(\phi_1)$ has density $61/192$, so that by Theorem \ref{thm:selmer-average} at least $61/384 \approx 15.88\%$ of twists $E_s/\mathbb{Q}$ have rank 0.  This expression also yields a bound of $\frac{1183}{864} \approx 1.369$ on the average rank of twists $E_s(\mathbb{Q})$ for $s$ squarefree.  Similarly, we compute that
\begin{align*}
\sum_{m\in \mathbb{Z}} \mu(T_m(\phi_2)) y^m 
	&= \left(\frac{1+y^{-1}}{2}\right) \left(\frac{y^{-1}}{6} + \frac{5}{6}\right) \left(\frac{27}{48}+\frac{21}{48}y^{-1}\right) \\
	&= \frac{7}{192} y^{-2} + \frac{17}{64}y^{-1} + \frac{89}{192} + \frac{15}{64} y.
\end{align*}
This yields that a larger proportion $89/384 \approx 23.18\%$ of twists $E_s/\mathbb{Q}$ have rank 0, and a smaller bound of $\frac{1043}{864} \approx 1.207$ on the average rank of $E_s(\mathbb{Q})$ for $s$ squarefree.

We can do much better by combining these isogenies, however.  In particular, we find that
\[
\sum_{m,n\in\mathbb{Z}} \mu(T_m(\phi_1) \cap T_n(\phi_2)) x^my^n = \frac{3}{64} xy^{-1} + \frac{15}{64} x^{-1}y + \frac{35}{192} x^{-2} + \frac{7}{192}y^{-2} + \frac{7}{32}x^{-1}y^{-1} + \frac{9}{32}.
\]
This shows that every squarefree $s$ is in either $T_0(\phi_1) \cup T_0(\phi_2)$ or $T_{-1}(\phi_1) \cup T_1(\phi_1) \cup T_{-1}(\phi_2) \cup T_1(\phi_2)$.  Thus, Corollary \ref{cor:ranks} yields that at least $25\%$ of twists $E_s/\mathbb{Q}$ have rank 0, at least $5/6$ have $3$-Selmer rank one, and that the average rank of $E_s(\mathbb{Q})$ for $s$ squarefree is at most $7/6$.  These rank bounds are the best one can hope for using only the methods of this paper.

Moreover, we are also able to exploit the isogenies $\phi_1$ and $\phi_2$ to produce $3$-torsion elements of Tate-Shafarevich groups.  In particular, the set $T_{-2}(\phi_1) \cap T_0(\phi_2)$ has density $35/192$.  By Proposition \ref{prop:sha-elements}, we find that a proportion $35/384$ of squarefree $s$ are such that $|\Sha(E^\prime_s)[3]| \geq 9$.  Similarly, we find that for at least $7/384$ of squarefree $s$, we have $|\Sha(E^{\prime\prime}_s)[3]| \geq 9$.

In fact, each of the curves $E$, $E^\prime$, and $E^{\prime\prime}$ also has a single rational two-torsion point, and hence a rational 2-isogeny.  The three additional curves that are the codomain of these 2-isogenies complete the isogeny class 14a, whose isogeny graph is given in Figure \ref{fig:14a}.

\begin{center}
\begin{figure}[h]
\begin{tikzpicture}[xscale=2.5,yscale=1.5]
\node (14a3) at (0,1) {14a3};
\node (14a1) at (1,1) {14a1};
\node (14a4) at (2,1) {14a4};
\node (14a5) at (0,0) {14a5};
\node (14a2) at (1,0) {14a2};
\node (14a6) at (2,0) {14a6};

\node at (0,1.1) [above] {\tiny $E^\prime$};
\node at (1,1.1) [above] {\tiny $E$};
\node at (2,1.1) [above] {\tiny $E^{\prime\prime}$};

\draw[<->] (14a3) edge (14a1);
\draw[<->] (14a1) edge (14a4);
\draw[<->] (14a5) edge (14a2);
\draw[<->] (14a2) edge (14a6);
\node at (0.5,1) [above] {3};
\node at (1.5,1) [above] {3};
\node at (0.5,0) [above] {3};
\node at (1.5,0) [above] {3};

\draw[<->] (14a3) edge (14a5);
\draw[<->] (14a1) edge (14a2);
\draw[<->] (14a4) edge (14a6);

\node at (0,0.5) [left] {2};
\node at (1,0.5) [left] {2};
\node at (2,0.5) [left] {2};
\end{tikzpicture}
\caption{The isogeny graph of the isogeny class 14a. \hfill \mbox{}}
\label{fig:14a}
\end{figure}
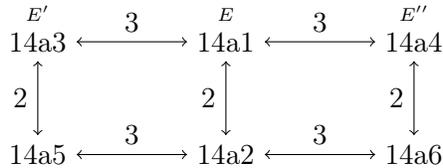
\end{center}

By an easy diagram chase, the global Selmer ratios of the two 3-isogenies in the top row are equal to those of the bottom row, so the analysis for the isogenies in the bottom row is identical to that of the top row.  In particular, exactly the same results hold for the proportion of twists with small rank, and exactly the same results hold on 3-torsion in Tate-Shafarevich groups for the curve 14a5 as do for $E^\prime$, and the same for 14a6 as do for $E^{\prime\prime}$.  In fact, exploiting the $2$-isogeny and Theorem \ref{thm:erdoskac}, it is possible to show that each curve in the family has a positive proportion of twists with arbitrarily large $2$-torsion in their Tate-Shafarevich groups.  

Unfortunately, it is not clear how to prove that any of these curves has a positive proportion of twists with an element of order six in $\Sha$, which is why these curves are the lone exceptional case in Theorem \ref{thm:18-isogeny}.  For example, to produce elements of order three in $\Sha(E^\prime_s)$, we used above that within $T_{-2}(\phi_1) \cap T_0(\phi_2)$, $50\%$ of twists $E^\prime_s(\mathbb{Q})$ have rank 0.  We also know by Theorem \ref{thm:erdoskac} that for $50\%$ of $s\in T_{-2}(\phi_1) \cap T_0(\phi_2)$, we have that $|\mathrm{Sel}_2(E^\prime_s)| \geq 2^{r_2}$ for any $r_2 \geq 0$.  To show that there is an element of order six in $\Sha$, we would need these two sets of density $1/2$ intersect, which we unfortunately see no way to guarantee.

\bibliographystyle{abbrv}
\bibliography{references-new}

\end{document}